\documentclass[10pt]{article}
\usepackage{amsmath,amsthm,amscd,amssymb,mathrsfs,setspace}
\usepackage{latexsym,epsf,epsfig}
\usepackage{epstopdf}

\usepackage[hmargin=3.5cm,vmargin=3.5cm]{geometry}
\usepackage{color}
\newcommand{\ds}{\displaystyle}

\newcommand{\reals}{\mathbb{R}}
\newcommand{\realstwo}{\mathbb{R}^2}
\newcommand{\realsthree}{\mathbb{R}^3}
\newcommand{\xb}{{\bf{x}}}

\newcommand{\Dx}{{\partial_x}}
\newcommand{\pd}{{\partial}}
\newcommand{\cA}{{\mathscr{A}}}
\newcommand{\cT}{{\mathscr{T}}}
\newcommand{\Dn}{\partial_{\nu}}
\newcommand{\Dz}{{\partial_z}}

\newcommand{\cE}{{\mathcal{E}}}
\newcommand{\cD}{\mathscr{D}}
\newcommand{\bA}{\mathbb{A}}
\newcommand{\bB}{\mathbb{B}}

\newcommand{\cH}{\mathscr{H}}
\newcommand{\bP}{\mathbb{P}}
\newcommand{\bX}{\mathbb{X}}
\newcommand{\cF}{\mathscr{F}}

\newcommand{\R}{\mathbb{R}}

\newcommand{\cU}{\mathcal{U}}
\newcommand{\Om}{{\Omega}}

\newcommand{\ga}{{\gamma}}

\newcommand{\la}{{\lambda}}

\theoremstyle{plain}
\newtheorem{theorem}{Theorem}[section]
\newtheorem{lemma}[theorem]{Lemma}
\newtheorem{proposition}[theorem]{Proposition}
\newtheorem{corollary}[theorem]{Corollary}

\newtheorem{condition}{Condition}[section]
\theoremstyle{remark}
\newtheorem{remark}{Remark}[section]
\numberwithin{equation}{section}
\numberwithin{theorem}{section}
\numberwithin{remark}{section}

\title{Nonlinear Plates Interacting with A Subsonic, Inviscid Flow  via  Kutta-Joukowski Interface  Conditions}
\date{\today}

 \author{\normalsize \begin{tabular}[t]{c@{\extracolsep{.8em}}c}
            Irena Lasiecka & Justin T. Webster \\
 \it University of Memphis    &\it Oregon State University \\
 \it Memphis, TN &\it Corvallis, OR\\
  \it lasiecka@memphis.edu &  \it websterj@math.oregonstate.edu
\end{tabular}}

\begin{document}
\maketitle

\begin{abstract} {\noindent We analyze the well-posedness of a flow-plate interaction considered in \cite{dowell,dowellnon}. Specifically, we consider the {\em Kutta-Joukowski} boundary conditions for the flow \cite{K1,K2,K4}, which ultimately give rise to a hyperbolic equation in the half-space (for the flow) with {\em mixed boundary conditions}. This boundary condition has been considered previously in the lower-dimensional interactions \cite{bal0,bal4}, and dramatically changes the properties of the flow-plate interaction and requisite analytical techniques.

 We present results on well-posedness of the fluid-structure interaction with the Kutta-Joukowsky flow conditions in force. The semigroup approach to the proof utilizes an abstract setup related to that in \cite{supersonic} but requires (1) the use of a Neumann-flow map to address a Zaremba type elliptic problem and (2) a trace regularity assumption on the acceleration potential of the flow. This assumption is linked to invertibility of singular integral operators which are analogous to the finite Hilbert transform in two dimensions. (We show the validity of this assumption when the model is reduced to a two dimensional flow interacting with a one dimensional structure; this requires microlocal techniques.) Our results link the analysis in \cite{supersonic} to that in \cite{bal0,bal4}.
\vskip.1cm
\noindent Key terms: flow-structure interaction, nonlinear plate,
 nonlinear semigroups, well-posedness, mixed boundary conditions, Possio integral equation, finite Hilbert transform. }
\end{abstract}

\section{Introduction}
\subsection{Motivation}
We aim to study the oscillations of a thin flexible plate interacting with an inviscid potential flow in which it is immersed. In the literature, many models have been suggested to accommodate various configurations and physical parameters. In this treatment we are concerned with analyzing the effect (from an infinite dimensional point of view) of the so called {\em Kutta-Joukowsky} (K-J) flow conditions in a flow-plate model of great recent interest. The K-J condition is stated in \cite{bal0,bal4} as taking ``a zero pressure jump off the wing and at the trailing edge"; in line with the analyses in \cite{bal0,bal4}, we take this to correspond to taking the {\em acceleration potential} of the flow to be zero outside the plate, in the plane of the plate. In this analysis we take {\em clamped plate boundary conditions} in order to focus on the abstract problems associated to the PDE analysis of the K-J condition. In fact, preliminary investigations indicate that free plate boundary conditions may better accommodate the K-J flow conditions---this is in line with certain engineering applications (e.g.  flag type models \cite{K1,K2}). However there are many technical challenges associated to PDE models of flow-plate interactions when the plate boundary conditions are not of homogeneous type. 

More specifically,  we address a flow-structure PDE model which describes the interactive dynamics between a plate and a surrounding potential flow (see, e.g., \cite{bolotin,dowellnon,dowell,dowell1} and the references therein). The novel feature of our analysis is the implementation of the K-J flow condition in the model considered in \cite{springer,webster,supersonic}. In the aforementioned analyses, a more straightforward (Neumann type) flow boundary condition is taken in the plane of the plate, in line with a standard {\em panel} configuration. To extend the analysis to more general (physical) configurations, well-posedness of the model must be established with the K-J flow condition of recent interest \cite{bal0,bal4} (see Section \ref{physicals} for more discussion). 
Our goals in the treatment are therefore (1) to make precise the K-J boundary condition in the three dimensional model found in \cite{dowellnon,springer} and provide a well-posedness result, and, (2) to relate two existing mathematical analyses of flow-plate models found in \cite{webster,jadea12,supersonic} and \cite{bal0,bal4}.

For simplicity in exposition, we first consider a linear plate in our arguments. Nonlinearity is then included, as it required for accuracy in modeling; however, it is nonessential to demonstrate the principal mechanisms at play with respect to the K-J condition. The considerations \cite{webster,jadea12,supersonic,springer} address the nonlinear aspects of the model in great detail. We address the nonlinear nature of the plate, and provide a discussion of the critical properties of the nonlinear model in Section \ref{physicals}.

Lastly, we note that these results were first reported in \cite{proceed}, without a complete proof. 
\subsection{Notation}
For the remainder of the text we write $\xb$ for $(x,y,z) \in \realsthree_+$ or $(x,y) \in \Omega \subset \realstwo_{\{(x,y)\}}$, as dictated by context. Norms $||\cdot||$ are taken to be $L_2(D)$ for the domain dictated by context. Inner products in $L_2(\realsthree_+)$ are written $(\cdot,\cdot)$, while inner products in $L_2(\R^2\equiv\pd\R^3_+)$ are written $<\cdot,\cdot>$. Also, $ H^s(D)$ will denote the Sobolev space of order $s$, defined on a domain $D$, and $H^s_0(D)$ denotes the closure of $C_0^{\infty}(D)$ in the $H^s(D)$ norm
which we denote by $\|\cdot\|_{H^s(D)}$ or $\|\cdot\|_{s,D}$. We make use of the standard notation for the trace of functions defined on $\realsthree_+$, i.e. for $\phi \in H^1(\realsthree_+)$, $\gamma[\phi]=\phi \big|_{z=0}$ is the trace of $\phi$ on the plane $\{\xb:z=0\}$.

\subsection{Model}
The model in consideration describes the interaction between a nonlinear plate with a field or flow of gas above it. (By antisymmetry, we need only consider the gas flow one side of the plate.) To describe the behavior of the gas we make use of the theory of potential flows (see, e.g., \cite{bolotin,dowell} and the references therein) which produce a perturbed wave equation for the flow velocity potential. The oscillatory behavior of the plate is governed by the second order (in time) Kirchoff plate equation. The principal point of interest in this treatment is the linear theory. In Section \ref{physicals} we briefly mention certain `physical' nonlinearities which are used in the modeling of the large oscillations of thin, flexible plates---so called \textit{large deflection theory}.

The environment we consider is $\realsthree_+=\{(x,y,z): z \ge 0\}$. The plate
 is modeled by  a bounded domain $\Omega \subset \reals^2_{\{(x,y)\}}=\{(x,y,z): z = 0\}$ with smooth boundary $\partial \Omega = \Gamma$.
 We take the plate to be clamped on all edges, immersed in an inviscid flow (over body) with velocity $U <1$ in the $x$-direction. Here we normalize $U=1$ to be Mach 1, i.e. $0 \le U <1$ corresponds to subsonic flows. 
 \begin{remark}
 Often in modeling panels, clamped or hinged boundary conditions are considered \cite{lagnese}. See the treatments \cite{webster&lasiecka,ACC,springer} for more details  on these plate boundary conditions and the corresponding coupling with the flow. 
 \end{remark}
 
\subsubsection{Plate Equation}
We analyze a general nonlinear plate equation, taken without rotational inertia. The coupling with the flow takes place in the external pressure term acting on the plate via the {\em acceleration potential} of the flow. 

The scalar function $u: \Om\times \reals_+ \to \reals$ represents the transverse displacement of the plate in the $z$-direction at the point $(x;y)$ at the moment $t$. We take the Kirchoff type plate with clamped boundary conditions:

\begin{equation}\label{plate}\begin{cases}
u_{tt}+\Delta^2u= p(\xb,t) & \text { in } \Om\times (0,T),\\
u(0)=u_0;~~u_t(0)=u_1,\\
u=\Dn u = 0 & \text{ on } \pd\Om\times (0,T).
\end{cases}
\end{equation}
The aerodynamical pressure  $p(\xb,t)$  represents the coupling with the flow
and will be given below.
\subsubsection{Flow Equation}
For the flow we make use of linearized
 potential theory, and \cite{BA62,bolotin,dowell1} the (perturbed) flow potential $\phi:\realsthree_+ \rightarrow \reals$ must satisfy the perturbed wave equation below:
\begin{equation}\label{flow}\begin{cases}
(\partial_t+U\partial_x)^2\phi=\Delta \phi & \text { in } \realsthree_+ \times (0,T),\\
\phi(0)=\phi_0;~~\phi_t(0)=\phi_1,\\
\Dz \phi = d(\xb,t)& \text{ on } \Omega \times (0,T)
\end{cases}
\end{equation}
Here, without  loss of generality, the density of the perturbed airflow is assumed to be equal to 1.  
We will be utilizing the {\em acceleration potential} $\psi \equiv \phi_t+U\phi_x$ as a state variable in what follows. 
\subsubsection{Coupling}
The strong coupling in the model takes place in the downwash term of the flow potential by taking $$d(\xb,t)=\big[(\partial_t+U\partial_x)u (\xb)\big]$$ for $\xb \in \Omega$. On $\partial \realsthree_+ \backslash \Omega$ we implement the so called {\em Kutta-Joukowski continuity condition} \cite{bal0,bal4,K1,K2,K4}: \begin{equation}\label{KJcondition}
\gamma[\phi_t+U\phi_x]=\gamma[\psi]=0, ~~\xb \in \realstwo \backslash \Omega.
\end{equation}
\begin{remark}
We distinguish this flow boundary condition from that which we refer to as the {\em standard flow} boundary condition, which is of full Neumann type:
The standard flow condition on the boundary for $\phi$ is \begin{equation}\label{standard}
\Dz \phi \big|_{z=0} = \begin{cases} (\partial_t+U\partial_x)u(t,\xb) & \xb \in \Omega\\ 0 & \xb \in \realstwo \backslash \Omega \end{cases} 
\end{equation} This boundary condition is presented in \cite{dowellnon,dowell1,dowell} and studied extensively in the subsonic case in \cite{webster,webster&lasiecka,jadea12,springer} and in the supersonic case in \cite{supersonic}.
\end{remark}
 The aerodynamical pressure on the surface of the plate is
of the form
\begin{equation}\label{aero-dyn-pr}
p(\xb,t)=\gamma[\psi]
\end{equation}
 in (\ref{plate}) above. This gives the fully coupled model:
\begin{equation}\label{flowplate}\begin{cases}
u_{tt}+\Delta^2u+f(u)= \gamma[\psi] & \text { in } \Om\times (0,T),\\
u(0)=u_0;~~u_t(0)=u_1,\\
u=\Dn u=0 & \text{ on } \pd\Om\times (0,T),\\
(\partial_t+U\partial_x)^2\phi=\Delta \phi & \text { in } \realsthree_+ \times (0,T),\\
\phi(0)=\phi_0;~~\phi_t(0)=\phi_1,\\
\Dz \phi = (\partial_t+U\partial_x)u & \text{ on } \Omega   \times (0,T).\\
(\partial_t+U\partial_x)  \phi = 0 & \text{ on } \realstwo \backslash \Omega \times (0,T)
\end{cases}
\end{equation}
\subsection{Energies and State Space}
In previous analyses for the {\em subsonic case}, without the K-J conditions in place, energies can be derived by applying standard plate and flow multipliers ($u_t$ and $\phi_t$, resp.) along with the standard boundary conditions above to obtain the energy relation for the system. This energy is well-behaved and bounded from below \cite{webster,jadea12}. 

In the case of supersonic flows, or in the case of subsonic flows {\em with} the K-J condition in place, this procedure leads to an energy which is unbounded (from below) for the flow component of the model. Hence, in line with the analysis of the supersonic flows (with standard flow boundary conditions) in \cite{supersonic}, we make use of the flow acceleration multiplier $(\partial_t+U\partial_x)\phi \equiv \psi$ for the {\em subsonic flow}, taken with the {\em K-J} flow boundary conditions. 
Thus for the flow dynamics, instead of $(\phi;\phi_t)$ we seek invariance for  $(\phi;\psi)$.

This leads to the following (formal) energies, arrived at via Green's Theorem when the multipliers $u_t$ and $\psi$ are applied to \eqref{flowplate}:
\begin{align}\label{energies}
E_{pl}(t) = \dfrac{1}{2}\big(||u_t||^2+||\Delta u||^2\big),~~
E_{fl}(t) =  \dfrac{1}{2}\big(||\psi||^2+||\nabla \phi||^2\big),
~~\cE(t) =  E_{pl}(t)+E_{fl}(t),
\end{align}
These energies provide the formal energy relation for the system (implementing the K-J conditions)
\begin{equation}\label{energyrelation}
\cE(t)+ U\int_0^t <u_x,\gamma[\psi]>_{\Omega} dt = \cE(0). \end{equation}
This energy relation provides motivations for viewing the dynamics as the sum of a generating piece and a `perturbation'.
\par
Finite energy constraints manifest themselves in the natural requirements on the functions $\phi$ and $u$:
\begin{equation}\label{flowreq}
\phi(\xb,t) \in C(0,T; H^1(\realsthree_+))\cap C^1(0,T;L_2(\realsthree_+)), \end{equation} \begin{equation}\label{platereq}
u(\xb,t) \in C(0,T; H_0^2(\Omega))\cap C^1(0,T;L_2(\Omega)).\end{equation}
The above finite energy constraints lead to the {\em finite energy space}: \begin{equation}\label{space-Y}
Y = Y_{fl} \times Y_{pl} \equiv \big(H^1(\realsthree_+)\times L_2(\realsthree_+)\big) \times \big(H_0^2(\Omega) \times L_2(\Omega)\big).
\end{equation}

We see above that the primary source of mathematical difficulty lies in interpreting and unravelling the energy relation in \eqref{energyrelation}.  In fact, this representation of the energies in (\ref{energies}) provides a good topological measure for the potential solution; however the energy balance is {\it lost} in (\ref{energyrelation}) when making use of the state variable $\psi$ and, in addition, the boundary term involves the traces of $ L_2$ solutions of the flow, which are possibly {\it not defined}  at all.
In view of these complications,  our approach is be based on  (1) developing a suitable   theory for the traces of the flow solutions (as in \cite{supersonic}); (2)  counteracting  the loss of energy balance relation. 
  \subsection{Definition of Solutions}\label{solutions}
  In the discussion below, we will encounter strong (classical), generalized (mild), and weak (variational) solutions.
In our analysis we will be making use of semigroup theory, hence we will work with \textit{generalized} solutions; these are strong limits of strong solutions. These solutions  satisfy an integral formulation of (\ref{flowplate}), and are called \textit{mild} by some authors. Topologically, generalized solution are finite energy solutions (i.e. $\cE(t) < \infty $ ) and are elements of the space $Y$. That is to say that they satisfy the regularity properties in (\ref{flowreq}), (\ref{platereq}). 
These solutions are obtained as strong limits of solutions which originate in the domain of the generator given  by (\ref{dom-bA-n}). 
Generalized solutions correspond to semigroup solutions for an initial datum outside of the domain of the generator. See \cite[Section 6.5.5]{springer} and \cite{webster} for more detailed discussion of solutions. In fact, in the context of linear problems (and the nonlinear models addressed in Section \ref{physicals}) generalized solutions are in fact {\em weak solutions}, i.e.,
they satisfy the corresponding variational forms
(see Definition 6.4.3 in \cite[Chapter 6]{springer}). 
\section{Statement of Results and Mathematical Context}
\subsection{Flow-Plate Analysis}
The original flow-plate interaction (for the clamped plate) was discussed numerically in \cite{bolotin,dowell,dowellnon}. We note that in these original considerations (in the panel configuration) the standard flow boundary condition was taken.  In the monograph \cite{springer}, the authors consider the flow-plate model abstractly in the case of subsonic flows, and show well-posedness. Later, the subsonic case was considered \cite{webster} and cast into a semigroup framework. Finally, in \cite{jadea12} a thorough treatment is given (again in the subsonic case) utilizing a viscosity method and obtaining hidden regularity of associated flow traces. In all of these cases, the standard boundary conditions and subsonic nature of the flow provide a {\em good} energy relation. However, the flow energy itself becomes degenerate for supersonic flows in this configuration. This degeneracy necessitates a ``change of state variable" ($\phi_t \to \psi$) in order to recover the validity of the flow energy; however, this comes at the cost of polluting the {\em energy relation}. Correspondingly, the well-posedness of the flow-plate system for supersonic flows was an open question until recently \cite{supersonic}. 
In this manuscript, the authors take the approach of recovering the dynamics (after the change of state variable) by viewing them as the sum of a generating component and a perturbation, where there perturbation corresponds to the ``energy polluting" term in the energy relation. Ultimately, to handle this perturbation, a suitable trace theory must be developed for the flow term $\gamma[\phi_t]$, which is a priori undefined ($\phi_t \in L_2(\realstwo)$). 

The key insight into the present analysis occurs at the level of the energy relation. For flow models which call for the K-J boundary condition, we may again utilize the state variable $\psi$ (rather than $\phi_t$); doing so, we arrive at the same energy relation in the case of supersonic flows (as discussed in the preceding paragraph). This indicates that the abstract approach (i.e. the decomposition of the dynamics) taken in the supersonic case  with standard boundary conditions will in fact accommodate the K-J boundary conditions, if a suitable trace theory can be developed for the flow. 

We pause here to mention another approach to the study of a similar flow-plate model; the author of \cite{bal0,bal4} considers a linear {\em wing}  immersed in a subsonic flow; the wing is taken to have a high aspect ratio thereby allowing for the suppression of the span variable, and reducing the analysis to individual chords normal to the span. By reducing the problem to a one dimensional analysis, many technical hangups are avoided, and Fourier-Laplace analysis is greatly simplified. Ultimately, the problem of well-posedness and $L_p$ regularity of solutions can be realized in the context of the classical {\em Possio integral problem} \cite{bal0,bal4}, involving the inverse Hilbert transform and analysis of Mikhlin multipliers \cite{mikhlin}. In our approach, we attempt to characterize our solution by similar means and point out how the two dimensional analysis greatly complicates matters and gives rise to singular integrals in higher dimensions. We also mention the confluence of our approach and the papers mentioned above in Remark \ref{Bal1}. 

\subsection{Statement of Main Results}
The main result of this paper provides existence, uniqueness and continuous dependence on the data  of finite energy solution. This result is obtained under a technical trace regularity condition imposed on aeroelastic potential, and stated in Condition \ref{le:FTR0}. In the case when the flow domain is two dimensional, the validity of this condition is proved herein (see also \cite{bal0}). We do not discuss this condition in detail here to avoid clouding the exposition.
Additionally, as discussed above, we focus on the linear theory in the arguments below. However, our analysis (as in \cite{supersonic,webster,jadea12,springer}) applies to the case of nonlinear plates. We have provided a description of the pertinent nonlinearities and critical properties in Section \ref{physicals}. 

The final result reads as follows: 
\begin{theorem}
With reference to the model (\ref{flowplate}), with  $0 \leq U < 1$:
\begin{enumerate}
\item Assuming $f: H_0^2(\Omega) \to L_2(\Omega)$ is {\em locally Lipschitz},
 \item
  Assuming the trace regularity Condition \ref{le:FTR0} holds for the aeroelastic potential $\psi = (\phi_t+U\phi_x)$,
  \end{enumerate}
  then there exists a unique finite  energy solution that is local in time. 
  
  This is to say, there exists $ T_0 > 0 $ 
  such that $(\phi, \phi_t, u, u_t ) \in C(0, T_0; Y ) $  for all initial data  $(\phi_0, \phi_1, u_0, u_1 ) \in  Y $. 
  This solution depend continuously on the initial data. \\
  If in addition, we take $f(u)$ to be the von Karman nonlinearity, the Berger nonlinearity, or Kirchoff type nonlinearity (Section 6.1 (1), (2), (3), resp.), then the solution above is global in time. In other words, the nonlinear dynamical system generates a continuous semigroup $ S_t$ on the space $Y$. 
\end{theorem}
Moreover, when we restrict to the lower dimensional case (and consider nonlinearities which are analogous to those listed in Section 6.1 in two dimensions) we have:
\begin{lemma}\label{c:1}
In the case when the dimension of $\Omega$ is one, Condition \ref{le:FTR0} is satisfied. Hence, in that case, any semiflow defined by (\ref{flowplate}) with nonlinear function subject to the hypotheses of Proposition \ref{abstractnonlin} generates a continuous semigroup.
\end{lemma}

\begin{remark}
The generation of semigroups for an arbitrary three dimensional flow is subjected to the validity of the trace Condition \ref{le:FTR0}.This, in turn, depends on invertibility properties of finite Riesz-type Transforms in two dimensional domains.   While it is believed that this property should be generically true, at the present stage this appears to be an open question in the analysis of singular integrals and depends critically on the geometry of $\Omega$ in two dimensions. 
\end{remark}

\subsection{Approach and Outline of Proof}
We briefly outline our approach:
\begin{enumerate}
\item As motivated by the supersonic analysis in \cite{supersonic}, we decompose the linear dynamics into a dissipative piece $\bA$ (unboxed below) and a perturbation piece $\bP$ (boxed below): \begin{equation}\label{sys1*}\begin{cases}
(\partial_t+U\partial_x)\phi =\psi& \text{ in } \realsthree_+ \times(0,T),
\\
(\partial_t+U\partial_x)\psi=\Delta \phi&\text{ in } \realsthree_+\times (0,T),
\\
\Dz \phi = \partial_tu +\framebox{$U\partial_xu $}& \text{ on }
\Omega \times (0,T),
\\
(\phi_t+U\phi_x) = 0 & \text{ on }
(\realstwo \backslash \Omega) \times (0,T),\\
u_{tt}+\Delta^2u=\gamma[\psi]&\text{ in }  \Om \times (0,T),\\
u=\Dn u = 0 & \text{ on } \pd\Om \times (0,T).\\
\end{cases}
\end{equation} We then proceed to show that $\bA$ (corresponding to the unboxed dynamics above) is $m$-dissipative on the state space.  Dissipativity is natural and built in  within the structure of the problem, while maximality requires analysis of the Zaremba problem (mixed flow boundary conditions).
\item To handle the ``perturbation" of the dynamics, $\bP$ (boxed) we cast the problem into an abstract boundary control \cite{redbook} framework following the analysis in \cite{supersonic}.
In order to achieve this, the critical ingredient in the proof is demonstrating  ``hidden"  boundary regularity for the acceleration potential $\psi$ of the flow. It will be shown that this component is an element of a negative Sobolev space
$L_2(0, T; H^{-1/2-\epsilon} (\Omega))$, based on an assumption about special integral transform related to the finite Hilbert transform.
The above regularity allows us to show that the term $<u_x,\gamma[\psi]>$ is well-defined via duality.
Consequently,  the  problem  with  the  ``perturbation" of the dynamics $\bP$,   can be recast
as an abstract boundary control problem with appropriate continuity properties of the control-to-state maps.

\item Then, to piece the operators together as $\bA+\bP$, we make use of variation of parameters with respect to the generation property of $\bA$ and appropriate dual setting. This yields an integral equation on the state space (interpreted via duality) which must be formally justified in our abstract framework. This is critically dependent upon point (2.) above. We then run a fixed point argument on the appropriate space to achieve a local-time solution for the fully linear Cauchy problem  representing formally  the evolution $y_t=(\bA+\bP)y \in Y$. In order to identify its generator, we apply Ball's theorem \cite{ball} which then yields global solutions.
\end{enumerate}
\section{Abstract Setup}
\subsection{Operators and Spaces}
We
 introduce the standard linear plate operator with clamped boundary conditions:
 $\mathscr{A}=\Delta^2$ with the domain
 \[
 \cD(\mathscr{A})=\{u\in H^4(\Omega): u=\Dn u = 0 \text{ on } \partial \Omega\}=(H^4\cap H_0^2)(\Omega).
   \]
  Additionally, $\mathscr{D}(\mathscr{A}^{1/2})= H_0^2(\Omega)
$. Take our state variable to be $$y\equiv(\phi, \psi; u, v) \in \big( H^1 (\realsthree_+) \times L_2(\realsthree_+)\big)\times\big( \cD(\cA^{1/2})\times L_2(\Omega)\big)\equiv Y.$$ We work with $\psi$ as an independent state variable, i.e., we are not explicitly taking $\psi = \phi_t+U\phi_x$ here. 

To build our abstract model, let us first define the operator $\bA: \cD(\bA) \subset Y \to Y$ by
\begin{equation}\label{op-A}
\bA \begin{pmatrix}
\phi\\
\psi\\ u\\ v
\end{pmatrix}=\begin{pmatrix}-U\partial_x \phi+\psi\\- U\partial_x\psi + \Delta \phi\\ v \\ -\cA u+ \gamma[\psi] \end{pmatrix}
\end{equation}
The domain of $\cD(\bA) $ is given by
\begin{equation}\label{dom-bA}
\cD(\bA) \equiv \left\{ y =  \begin{pmatrix}
\phi\\
\psi\\ u\\ v
\end{pmatrix}\in Y\; \left| \begin{array}{l}
-U \Dx \phi + \psi \in H^1(\R^3_+),~\\ -U \Dx \psi  + \Delta \phi  \in L_2(\R^3_+), \\
\psi =0 \text{ on } \realstwo\backslash\Omega, ~~\Dn \phi = -v , \text{ in} ~ \Omega \\
v \in \cD(\cA^{1/2} )= H_0^2(\Omega),~\\
-\cA u + \gamma  [\psi] \in L_2(\Omega) \end{array} \right. \right\}
\end{equation}
The operator $\bA$ will be the foundation of our abstract setup, and ultimately the dynamics of the evolution in \eqref{flowplate} can be represented through $\bA$. 

As a first step, we show that  $\bA$ is m-dissipative. 
 
 \subsection{Semigroup Generation of $\bA$}
\begin{theorem}
The operator $\bA$ is  $m$-dissipative on $Y$.  Hence, via the Lumer-Philips theorem, it generates a C$_0$ semigroup of contractions. 
\end{theorem}
\begin{proof} 
\subsubsection{Dissipativity}
We employ the following inner product on our state space: for $y,~\hat y \in Y$

$$ (( y, \hat{y}))_Y \equiv  ( \nabla \phi , \nabla \hat{\phi} )_{  \realsthree_+}  + (\psi, \hat{\psi})_{  \realsthree_+}   + < \cA^{1/2} u, \cA^{1/2} v >_{\Om}  + < v, \hat{v} >_{\Om} $$

With this inner product, we have for any $y\in \cD(\bA)$:

$$(( \bA y, y))_Y = -U (\nabla \phi_x, \nabla {\phi} )_{  \realsthree_+}  + (\nabla \psi, \nabla \phi )_{  \realsthree_+}  - U (\psi_x , \psi ) - ( \nabla \phi, \nabla \psi)
+ < \Dn \phi , \psi >_{\realstwo} + $$ $$ <\cA^{1/2} u , \cA^{1/2} v >_{\Om} - < \cA u, v >_{\Om} + < \gamma[\psi], v > $$
Implementing the K-J  boundary conditions and the coupling:
$$ < \Dn \phi, \psi>_{\realstwo} = - < v, \gamma[\psi]>_{\Omega} + < \Dn \phi, \psi>_{ \realstwo \backslash  \Omega} $$
yields     $(( \bA y, y))_Y =0 $ for any $y \in \cD(\bA)$. \\

\subsubsection{Maximality}

We must verify that $\mathscr R(\bA + \lambda I)=Y$ for some $\lambda \geq 0 $. 
We will take the system with $\lambda = 0$ first for simplicity.
Given $(f_1,f_2;g_1,g_2) \in Y$, we consider:
\begin{eqnarray}
-U \phi_x + \psi = f_1 \in H^1(  \realsthree_+) \\
 -U \psi_x + \Delta \phi = f_2\in L_2(  \realsthree_+) \\
v = g_1 \in \cD(\cA^{1/2} ) \\
- \cA u + \gamma[\psi] = g_2 \in L_2(\Om),
\end{eqnarray}
with boundary conditions:
\begin{equation}\begin{cases}\Dn \phi = -v = -g_1 \in H^2(\Omega)\\ 
\psi =0 \text{ in }~ \realstwo\backslash\Omega.\end{cases}\end{equation}
The last condition is equivalent to
$$ U \phi_x = - f_1 \in H^{1/2} (\realstwo\setminus \Om),$$ via the trace theorem.

As we are in the subsonic case, $ U < 1$, the above problem is strongly elliptic. Let us now solve for  $\tilde{\phi } = \phi_x $. We differentiate the flow relations above and the first boundary condition in $x$, and combine the flow equations to arrive at the elliptic problem \begin{equation}
\begin{cases}\Delta_U\tilde \phi = f_{2,x}+Uf_{1,x} \in L_2(\realsthree_+)\\
\Dn \tilde\phi =-g_{1,x} \in H^1(\Omega)\\
\tilde \phi=-\frac{1}{U}f_1 \in H^{1/2}(\realstwo\backslash \Omega)\end{cases}
\end{equation}
where $\Delta_U=\Delta-U^2\partial_{x}^2$.

This is a mixed (Zaremba) elliptic problem \cite{savare}. 
We recall,  with $ \Gamma_1 \cup \Gamma_2 = \partial D $ , where $D$ is a bounded and smooth domain,  the solution to 
\begin{equation}\label{zar}\begin{cases}
\Delta w \in L_2(D)  \\
\Dn w \in H^{-1/2 + \theta}(\Gamma_1)  \\
w \in H^{1+ \theta } (\Gamma_2) \end{cases}
\end{equation}
enjoys the regularity $ w \in H^{1+ \theta} (D)$ where $ \theta \in (-1/2, 1/2) $. 
Thus the maximal spatial regularity possess $ 3/2 - \epsilon $ derivatives. 
The above result is almost optimal, since it is known that in general when $\Gamma_0 $ and $\Gamma_1$ 
do not meet  at  an appropriate acute  angle, one has that $ w \notin B^{3/2}_{2, q}(D) $ for all $ q < \infty $; on the other hand 
$B_{2,\infty}^{3/2}(D)  \subset H^{3/2- \epsilon}(D) $ (where $B^s_{p,q}$ is the appropriately defined Besov space on $D$). 

An application of the above result with $\theta =0 $    yields the  solution $$ \tilde{\phi}  = \phi_x \in H^1(  \realsthree_+ ). $$ Returning to the equation, we recover $\psi$: $ \psi = U \phi_x + f_1 \in H^1(  \realsthree_+) $ and hence,
 $\psi \in H^1(  \realsthree_+) $. With $\psi\in H^1(\realsthree_+)$ in hand (and hence $\gamma[\psi] \in H^{1/2}(\Omega)$), solving for $(u,v)$ is standard. In addition, having solved for $\phi_x$, we may then specify that $\Delta \phi = f_2+\psi_x \in L_2(\realsthree_+)$, with appropriate boundary conditions. We must verify that this is valid by recovering $\phi \in H^1(\realsthree_+)$. \begin{remark} Note that (from the regularity of the flow equations and mixed boundary conditions) we {\em will not obtain} $ \phi \in H^2(\realsthree_+)$, demonstrating that the resolvent operator, in this case, is not compact. \end{remark}
 
 To see that $\phi \in H^1 (\realsthree_+) $, we proceed as follows: let $\lambda > 0 $ and consider the equation for $(\bA +\lambda I)y=(f_1,f_2;g_1,g_2) \in Y$, where $y=(\phi,\psi;u,v)$ is a solution (as obtained above for the case $\lambda=0$):
 \begin{eqnarray}
-U \phi_x + \psi  +\lambda \phi = f_1 \in H^1(  \realsthree_+) \\
 -U \psi_x + \Delta \phi  + \lambda \psi = f_2\in L_2(  \realsthree_+) \\
v  + \lambda u = g_1 \in \cD(\cA^{1/2} ) \\
- \cA u + \gamma[\psi] + \lambda v  = g_2 \in L_2(\Om)
\end{eqnarray}
Applying $\nabla $ to both sides of first equation, multiplying by $\nabla \phi $, and integrating gives the relation
$$ (\nabla \psi, \nabla \phi) + \lambda ||\nabla \phi||^2 = (\nabla f_1, \nabla \phi ) $$ 
Multiplying the second equation by $\psi$, and integrating by parts (with the boundary conditions) gives
$$ - (\nabla \phi, \nabla \psi) - < v, \psi>  + \lambda ||\psi ||^2 = (f_2, \psi) .$$
Adding the two equations and bounding yields the following a priori estimate on $(\phi,\psi)$:
$$\lambda ||\nabla \phi||^2 + \lambda ||\psi||^2 \leq C \Big[ ||v||_{1/2, \Omega} ||\psi||_{-1/2, \Omega } + ||f_1||^2_{1,   \realsthree_+} + ||f_2||^2_{0,   \realsthree_+ }\Big] .$$ In addition, we have the standard bound for the plate components $(u,v)$ of the system: $$\lambda ||\cA^{1/2}u||^2+\lambda ||v||^2 \le C\big[||v||_{1/2}||\psi||_{-1/2,\Omega}+||\cA^{1/2}g_1||^2+||g_2||^2 \big] .$$ 
We note that from the equations we recover $v \in \cD(\cA^{1/2})$. Moreover, \begin{align}||v||_{1/2,\Omega}^2 \le &C\big[\lambda^2||u||^2_{1/2,\Omega}+||g_1||^2_{1/2,\Omega}\big] \\ \le&C\big[\lambda^2||\cA^{1/2}u||^2+||g_1||^2_{2,\Omega}\big]. \end{align}  Thus, $$||v||_{1/2}||\psi||_{-1/2,\Omega} \le \epsilon ||v||_{1/2}^2+C_{\epsilon}||\psi||_{-1/2,\Omega}^2$$ for all $\epsilon>0$, and so $$||v||_{1/2}||\psi||_{-1/2} \le \epsilon ||\cA^{1/2}u||^2+C_{\epsilon,\lambda}||g_1||^2_{2,\Omega}+C_{\epsilon}||\psi||_{0,\realsthree_+}^2.$$
Finally, taking $\lambda$ sufficiently large, we have the final a priori estimate on the solution:
$$\lambda ||(\phi,\psi;u,v)||_Y^2 \le C||(f_1,f_2;g_1,g_2)||_Y^2.$$
 \end{proof}
 In addition, we note from the proof of the m-dissipativity of $\bA$ above, that $-\bA$ is also m-dissipative; indeed, $((-\bA y,y))_Y=0$.  The proof of maximality (the corresponding estimates) do not depend on the sign of $\bA$, owing to inherent cancellations and the structure of the static flow problem. Thus, with both $\pm \bA$ m-dissipative, we have \cite[Cor. 2.4.11]{cas-har98}: \begin{corollary}
The operator $\bA$ is skew-adjoint on $Y$ and generates a C$_0$ group of isometries.
 \end{corollary}
 \subsection{Abstract Representation of Boundary Conditions}
  In order to encode the flow boundary conditions abstractly into our operator representation of the evolution, we introduce the {\em flow-Neumann map}
 defined for the flow operator $$\bA_0 \begin{pmatrix}\phi\\\psi\end{pmatrix} \equiv   \begin{pmatrix}- U \phi_x + \psi\\ - U \psi_x + \Delta \phi \end{pmatrix} , $$ 
 with $$\cD(\bA_0)= \{ (\phi, \psi) \in Y_f \equiv H^1(  \realsthree_+) \times L_2(  \realsthree_+)\big|~  - U \phi_x + \psi \in H^1(  \realsthree_+),  ~~U \psi_x + \Delta \phi\in L_2 (  \realsthree_+), $$ $$\Dn \phi =0 \text{ on }~ \Omega, ~\psi =0  ~\text{ in }~\realstwo\backslash\Omega \}.$$
 By the argument utilized in the proof of maximality above, we notice that the membership in $\cD(\bA_0)$ implies that 
 $ \phi_x \in H^1(  \realsthree_+)$, $\phi  \in H^1(  \realsthree_+) $, and $\psi \in H^1(  \realsthree_+).$
 In addition, as before (neglecting the plate components), we have that  the operators $\pm\bA_0$ are $m$-dissipative on $Y_f$. 
This indicates that $\bA_0$ is skew-adjoint; we demonstrate the symmetric action below:
 $$(( \bA_0 y, \hat{y}))_{Y_f} = -U (\nabla \phi_x, \nabla\hat {\phi} )_{  \realsthree_+}  + (\nabla \psi, \nabla \hat{\phi} )_{  \realsthree_+}  - U (\psi_x ,\hat{ \psi} )_{  \realsthree_+} - ( \nabla \phi, \nabla \hat{\psi})_{  \realsthree_+} 
+ < \Dn \phi , \gamma[\hat{\psi} ]>_{\realstwo}  $$ 
We first utilize the boundary conditions:
$$ < \Dn \phi,\gamma[\hat{ \psi}]>_{\realstwo} =  < \Dn \phi, \gamma [\hat{\psi}]>_{\Omega} + < \Dn \phi,\gamma[\hat{ \psi}]>_{ \realstwo \backslash  \Omega}  =0.$$
Then, integrating by parts in $x$, using Green's formula, and once more utilizing the boundary conditions: 
  $$(( \bA_0 y, \hat{y}))_{Y_f} = U (\nabla \phi, \nabla\hat {\phi}_x )_{  \realsthree_+}  + (\nabla \psi, \nabla \hat{\phi} )_{  \realsthree_+}  + U (\psi ,\hat{ \psi}_x ) - ( \nabla \phi, \nabla \hat{\psi})$$
$$  = (\nabla \phi, U \nabla\hat {\phi}_x )_{  \realsthree_+}  - (\nabla \psi, \Delta  \hat{\phi} )_{  \realsthree_+}  + U (\psi ,\hat{ \psi}_x ) - ( \nabla \phi, \nabla \hat{\psi}).$$
Hence $$ (( \bA_0 y, \hat{y}))_{Y_f} =- ((y , \bA_0 \hat{y} ))_{Y_f}  ~\text{ for }~ y, ~\hat{y} \in \cD(\bA_0).$$

With a help of the flow map $\bA_0$,  we define Neumann-flow map as follows:
$$ N : L_2(\realstwo) \rightarrow Y_f $$ given by 
$$(\phi, \psi ) = Ng , ~~\text{ iff }$$ $$-U \phi_x + \psi + \phi  =0 ~\text{ and } ~- U \psi_x + \Delta \phi + \psi  =0~ \text{ in }~   \realsthree_+, ~\text{ with }~
\psi =0 ~\text{ in }~\realstwo\backslash\Omega, ~\text{ and }~\Dn \phi =-g  ~\text{ in }~ \Omega.$$

We then consider the associated regularity of the map $N$. Note that the Neumann map  is associated with the matrix operator $\bA_0 $ rather than the usual harmonic extensions associated with a scalar elliptic operator. This difference is due to the fact that K-J conditions affect both the flow and the  aeroelastic potential. 
In order to describe the regularity of the $N$ map we shall use the following anisotropic function spaces: 
$$H^{r,s}(D) = \{ f \in H^s(D),  \frac{\partial^r  f}{\partial x^r } \in H^s(D) \} $$ 
These spaces are subspaces of $H^s(D) $ with the additional information on regularity in $x$-direction.

\begin{lemma}
$$N\in \mathscr{L} \big(H^{1,-1/2}(\Omega) \rightarrow H^{1,1}(  \realsthree_+) \times H^{0,1}(  \realsthree_+)\big), $$ 
where $g\in H^{1,-1/2}(\Omega) $ indicates that $\partial_xg \in H^{-1/2}(\Omega)$ 
\end{lemma}
\begin{proof}
The regularity of the Neuman-flow problem is related to the Zaremba elliptic problem (as discussed above). 
First, we take $Ng = (\phi, \psi) $, where  $\psi = U \phi_x - \phi  $ and 
$$ - U^2 \phi_{xx} + \Delta \phi  - \phi +2U\phi_x=0,$$
with the mixed boundary conditions:
$$ \Dn \phi =-g, ~\text{ in }~\Omega , ~~\psi =0 ~\text{ in }~\realstwo \backslash\Omega. $$ 
Thus $\tilde{\phi} = \phi_x $ satisfies the problem
$$ - U^2 \tilde{\phi}_{xx} + \Delta \tilde{\phi}  - \tilde{\phi}+2U\tilde{\phi}_x  =0 ~\text{ in }~  \realsthree_+  $$
$$\tilde{\phi} =0 ~\text{ in }~\realstwo \backslash \Omega, ~\Dn \tilde{\phi} =-g_x ~\Omega. $$
 This is Zaremba mixed problem, which then yields (with $g_x \in H^{-1/2} (\Omega)$) the solution $ \tilde{\phi} \in H^1(  \realsthree_+). $ Consequently $\psi = U \tilde{\phi} \in H^1(  \realsthree_+).$ Finally 
 $\phi = U \phi_x - \psi \in H^1(\realsthree_+ ) $. 
 \end{proof}
 \begin{remark}
 We again emphasize the dependency of the above result on the strong ellipticity of the operator $\Delta_U=\Delta-U^2\partial^2_x$. The operator enjoys this property only when considering subsonic flows $0\le U <1$.
 \end{remark}
 Our next result identifies $N^*[ \bA_0^*+I] $ with a trace operator.  Here, the adjoint taken is with respect to the $L_2(\Omega) \rightarrow Y_f$ topology. This is reminiscent of a classical Neumann map: 
 \begin{lemma}\label{duality}
 Let $(\phi, \psi) \in \cD(\bA_0^*) $. Then 
 $N^*[ \bA_0^* + I ]  (\phi, \psi) = \gamma [\psi] $.  
 \end{lemma}

 \begin{proof}
 
 Let $Ng = (\hat{\phi}, \hat{ \psi} ) $. This means: $$ - U \hat{\phi}_x + \hat \psi + \hat \phi =0,~~ - U \hat{\psi}_x + \Delta \hat{\phi} + \hat \psi =0,~~ \Dn \hat{\phi} =-g \text{ on } \Omega,~~ \hat{\psi} =0 \text{ on } \realstwo \backslash \Omega$$
\begin{align*}
 < N^*[ \bA^*_0 + I ]  (\phi, \psi) , g >_{\Omega} =&~ ([\bA^*_0 + I ] (\phi, \psi) ,N g )_{Y_f}  \\ 
 =&~ \big(\nabla(U \phi_x -  \psi  + \phi),  \nabla \hat{\phi}\big)_{  \realsthree_+} +
(U\psi_x - \Delta \phi + \psi , \hat{\psi} )_{  \realsthree_+}  \\
=&~ (U\nabla \phi_x,\nabla \hat \phi)_{\realsthree_+}-(\nabla \psi, \nabla \hat \phi)_{\realsthree_+}+ (\nabla \phi, \nabla \hat \phi)\\
&+(U\psi_x, \hat{\psi} )_{  \realsthree_+}  - (\Delta \phi, \hat{\psi} )_{  \realsthree_+}  + (\psi , \hat{\psi} )_{  \realsthree_+} \end{align*} 
We utilize Green's theorem in the second and fifth terms:
\begin{align*} 
\phantom{< N^*[ \bA^*_0 + I ]  (\phi, \psi) , g >_{\Omega}}=&~ (U\nabla \phi_x,\nabla \hat \phi)_{\realsthree_+}+( \psi, \Delta \hat \phi)_{\realsthree_+}+ (\nabla \phi, \nabla \hat \phi)\\
&+(U\psi_x, \hat{\psi} )_{  \realsthree_+}  + (\nabla \phi, \nabla \hat{\psi} )_{  \realsthree_+}  + (\psi , \hat{\psi} )_{  \realsthree_+}\\
&-<\gamma[\psi],\Dn \hat \phi>_{\realstwo}+<\Dn \phi , \gamma[\hat \psi]>_{\realstwo}
\end{align*}
We may simplify the boundary terms using the boundary conditions for $(\hat \phi, \hat \psi)$ and the fact that $(\phi,\psi) \in \cD(\bA_0^*)=\cD(\bA_0)$.  \begin{align*}
\phantom{< N^*[ \bA^*_0 + I ]  (\phi, \psi) , g >_{\Omega}}=&~ (U\nabla \phi_x,\nabla \hat \phi)_{\realsthree_+}+( \psi, \Delta \hat \phi)_{\realsthree_+}+ (\nabla \phi, \nabla \hat \phi)\\
&+(U\psi_x, \hat{\psi} )_{  \realsthree_+}  + (\nabla \phi, \nabla \hat{\psi} )_{  \realsthree_+}  + (\psi , \hat{\psi} )_{  \realsthree_+}\\
&+<\gamma[\psi],g>_{\Omega}
\end{align*}

At this point we utilize the relations from the $N$ map in the second and fifth terms:
 \begin{align*}
\phantom{< N^*[ \bA^*_0 + I ]  (\phi, \psi) , g >_{\Omega}}=&~ (U\nabla \phi_x,\nabla \hat \phi)_{\realsthree_+}+( \psi, (U\hat{\psi_x}-\hat \psi))_{\realsthree_+}+ (\nabla \phi, \nabla \hat \phi)\\
&+(U\psi_x, \hat{\psi} )_{  \realsthree_+}  + (\nabla \phi, \nabla (U\hat\phi_x-\hat \phi) )_{  \realsthree_+}  + (\psi , \hat \psi )_{  \realsthree_+}\\
&+<\gamma[\psi],g>_{\Omega}\\
=& <\gamma[\psi],g>_{\Omega},
\end{align*}
where in the last line we have utilized integrated by parts multiple times and used the fact that $\phi \in H^1(\realsthree)$ and $\psi$ compactly supported on $\realstwo$.
\end{proof}
With the introduced notation we can express the flow-structure operator as 
 \begin{equation}\label{op-AS}
\bA \begin{pmatrix}
\phi\\
\psi\\ u\\ v
\end{pmatrix}=\begin{pmatrix} \cdot\\\bA_0  \Big[\begin{pmatrix}\phi\\ \psi\end{pmatrix}  - N v \Big]  - Nv 
 \\ v \\ -\cA u + N^* (\bA_0^*  + I )  \begin{pmatrix}\phi\\\psi \end{pmatrix} \end{pmatrix}.
\end{equation}
This new representation of $\bA$ encodes the boundary conditions, and further reveals the antisymmetric structure of the problem.

\subsection{Cauchy Problem}
Having established that $\bA$ is $m$-dissipative, the Cauchy problem
\begin{equation}\label{eq-trankt-0}
y_t=\bA y, ~y(0)=y_0 \in Y
\end{equation}
 is well-posed on $Y$. 
 
The dynamics of the {\it original } fluid-structure interaction in \eqref{flowplate} can be re-written (taking into account the action and domain of $\bA$) as
\begin{equation}\label{op-ASf}
\bB_L   \begin{pmatrix}
\phi\\
\psi\\ u\\ v
\end{pmatrix}  = \bA \begin{pmatrix}
\phi\\
\psi\\ u\\ v
\end{pmatrix}  + \bP \begin{pmatrix}  \phi \\ \psi\\ u\\ v \end{pmatrix} ,
\end{equation}
where $\bP$ is what remains of the dynamics in \eqref{flowplate} which is not captured by $\bA$.
This allows us to treat the problem of well-posedness within the framework of ``unbounded trace perturbations",
where the perturbation in question becomes 
\begin{equation}\label{op-ASP}
\bP \begin{pmatrix}
\phi\\
\psi\\ u\\ v
\end{pmatrix}=\begin{pmatrix} \cdot\\-U (\bA_0+I)N u_x 
 \\ 0 \\ 0  \end{pmatrix}
\end{equation}
Here $(\bA_0+I)N $ is defined via duality (using its adjoint expression) via Lemma \ref{duality}.

We now verify that $\bA+\bP$ (computed formally) fully encodes the dynamics of \eqref{flowplate}. 
\begin{align*}
(\bA+\bP)y=&\begin{pmatrix}\cdot\\ \bA_0\Big[\begin{pmatrix} \phi \\ \psi\end{pmatrix}-N(v+Uu_x)\Big]-N(v+Uu_x) \\ v \\ -\cA+N^*(\bA^*_0+I)\begin{pmatrix}\phi\\\psi \end{pmatrix} \end{pmatrix}
\end{align*}
That the plate components are correct is standard. We focus on the flow component: 
Let $\begin{pmatrix} \hat \phi \\ \hat \psi \end{pmatrix} = N(v+Uu_x)$. This implies that 
\begin{equation}\label{piece}-U\hat \phi_x+\hat \psi = - \hat\phi,~~-U\hat\psi_x+\Delta \hat\phi = -\hat\psi, ~~\Dn \hat\phi = -(v+Uu_x)~\text{ on } \Omega,~~\hat\psi = 0 \text{ on } \realstwo\backslash \Omega.\end{equation}
Hence
\begin{align*}
\bA_0\Big[\begin{pmatrix} \phi \\ \psi\end{pmatrix}-N(v+Uu_x)\Big]-N(v+Uu_x) =& ~\bA_0\Big[ \begin{pmatrix} \phi \\ \psi \end{pmatrix}-\begin{pmatrix} \hat \phi \\ \hat \psi\end{pmatrix}\Big] -\begin{pmatrix} \hat \phi \\ \hat \psi\end{pmatrix}\\
=&~ \begin{pmatrix} -U(\phi+\hat \phi)_x-(\psi+\hat \psi) \\ -U(\psi+\hat \psi)_x-\Delta(\phi+\hat \phi)\end{pmatrix}-\begin{pmatrix} \hat \phi \\ \hat \psi \end{pmatrix} \\
=&~\begin{pmatrix}-U\phi_x+\psi \\ -U\psi_x+\Delta \phi \end{pmatrix} -\begin{pmatrix} -U\hat \phi_x+\hat \psi \\ -U\hat \psi_x+\Delta \hat \phi \end{pmatrix}-\begin{pmatrix} \hat \phi \\ \hat \psi \end{pmatrix}\\
=&~\begin{pmatrix}-U\phi_x+\psi \\ -U\psi_x+\Delta \phi \end{pmatrix} 
\end{align*} where we have used \eqref{piece} in the last line to make the cancellation.

\section{Generation for the Full Dynamics: $\bA+\bP$}
We would like to recast the full dynamics of the problem in \eqref{flowplate} as a Cauchy problem in terms of the operator $\bA$. To do this, we define an operator $\bP:Y \to \mathscr{R}(\bP)$ as follows:
\begin{equation}\label{op-P}
\bP\begin{pmatrix}\phi\\\psi\\u\\v \end{pmatrix}= \bP_{\#}[u]\equiv\begin{pmatrix}0\\-U\bA_0N\partial_x u\\0\\0\end{pmatrix}
\end{equation}
Specifically, the problem in (\ref{flowplate}) has the abstract Cauchy formulation:
\begin{equation*}
y_t = (\bA +\bP)y, ~y(0)=y_0,
\end{equation*}
where $y_0 \in Y$ will produce semigroup (mild) solutions to the corresponding integral equation, and $y_0 \in \cD(\bA)$ will produce classical solutions. To find solutions to this problem, we will consider a fixed point argument, which necessitates interpreting and solving the following inhomogeneous problem,  and then producing the corresponding estimate on the solution:
\begin{equation} \label{inhomcauchy}
y_t = \bA y +\bP_\# \overline{u}, ~t>0, ~y(0)=y_0,
\end{equation}
for a given $\overline{u}$. To do so, we must understand how $\bP$ acts on $Y$
(and thus $\bP_\#$ on $H_0^2(\Om)$).
\par
To motivate the following discussion, consider for $y \in Y$ and $z= (\overline\phi, \overline\psi;  \overline u, \overline v)$ the formal calculus (with $Y$ as the pivot space)
\begin{align}\label{p-y-z}
(\bP y, z)_Y =& (\bP_{\#}[u],z)_Y
= -U((\bA_0+I)N\partial_x u,\overline \psi)
=-U<\partial_x u, \gamma[\overline \psi]>.
\end{align}
Hence, interpreting the operator $\bP$ (via duality) is contingent upon the ability to make sense of $\gamma[\overline \psi]$, which \textit{can} be done if $\gamma[\overline \psi] \in H^{-1+\epsilon}(\Omega)$ (since $u_x \in H^1(\Omega)$). In what follows, we show a trace estimate on $\psi$ (for semigroup solutions) of (\ref{inhomcauchy}) allows us to justify the program outlined above.
\par
We now state the trace regularity which is required for us to continue the abstract analysis of the dynamics. We relegate a discussion (and the corresponding proof in one dimension) to Section \ref{tracereg}.
  \begin{condition}[{\bf Flow Trace Regularity}]\label{le:FTR0}
   $\phi(\xb,t)$  satisfies \eqref{flow} and \eqref{KJcondition}, then $$\partial_t\gamma[\phi],~~ \partial_x\gamma[\phi]  \in L_2(0,T;H^{-1/2-\epsilon}(\realstwo))~~~~\forall\, T>0.
$$
Moreover,  with $\psi = \phi_t + U \phi_x $ we have
\begin{equation}\label{trace-reg-est-M0}
\int_0^T\|\gamma[\psi](t)\|^2_{H^{-1/2 -\epsilon} (\realstwo)}dt\le C_T\left(
E_{fl}(0)+
 \int_0^T\| \Dn \phi(t) \|_{\Omega}^2dt\right)
\end{equation}
\end{condition}
\noindent In what follows we implement microlocal analysis and reduce this theorem to a statement about integral integral transforms in $\realstwo$ analogous to the finite Hilbert transform. This theorem is critical for the arguments to follow concerning the perturbation $\bP$ and general abstract approach. We note, the above result holds for \textit{any} flow solver---we will be applying this result in the case where $\Dn \phi = - v \in C(0,T;L_2(\Omega))$ coming from a semigroup solution generated by $\bA$.
\par
At this stage we follow the approach taken in \cite{supersonic} by interpreting the variation of parameters formula for $\overline{u} \in C(\R_+; H^2_0(\Om))$
\begin{equation}\label{solution1}
y(t)=e^{\bA t}y_0 + \int_0^t e^{\bA(t-s)}\bP_{\#} [\overline{u}(s)] ds.
\end{equation}
by writing (with some $\la\in\R$, $\la\neq 0$):
\begin{equation}\label{solution2}
y(t)=e^{\bA t}y_0 + (\lambda  -\bA)\int_0^t e^{\bA(t-s)}(\lambda  - \bA)^{-1}\bP_{\#} [\overline{u}(s)]  ds.
\end{equation}
 We initially take this solution in $[\cD(\bA^*)]'=[\cD(\bA)]'$ (via the skew-adjointness of $\bA$), i.e., by considering the solution $y(t)$ in (\ref{solution2}) above acting on an element of $\cD(\bA^*)$.

 \subsection{Abstract Semigroup Convolution}
 We now recall a key theorem in the theory of abstract boundary control. This theorem will allow us to view the operator $\bP$ (mapping $Y$ outside of itself) as an unbounded perturbation. To do this, we critically implement the trace regularity theorem above in \eqref{le:FTR0}.  
 \par
 Let  $X$ and $\cU$ be  reflexive Banach spaces. We assume that
 \begin{enumerate}
\item[(C1)] $A$ is a linear operator which generates a strongly continuous semigroup $e^{At}$ on $X$.
\item[(C2)] $B$ is a linear continuous operator from $\cU$ to $[\cD(A^*)]'$ (duality with respect to the pivot space $X$), or equivalently, $(\lambda - A)^{-1}B \in \mathscr{L}(\cU,X)$ for all $\lambda \in \rho(A)$.
 \end{enumerate}
  For fixed $0<T<\infty$ and $u\in L_1(0,T;\cU)$ we  define  the convolution operator
  $$
  (Lu)(t) \equiv \int_0^t e^{A(t-s)}Bu(s) ~ds,~~0 \le t \le T,
  $$
  corresponding to the mild solution
  $$
  x(t) =e^{At}x_0+(Lu)(t),~~0 \le t \le T,
  $$
  of the abstract inhomogeneous equation
  $$
  x_t = Ax + Bu \in [\cD(A^*)]',~x(0)=x_0,
  $$
  with the input function $Bu(t)$.

\begin{theorem}[{\bf Inhomogeneous Abstract Equations}]\label{dualityequiv}
Let $X$ and $\cU$ be reflexive Banach spaces and
the  conditions in $(C1)$ and $(C2)$ be in force.
Then
\begin{enumerate}
\item The semigroup $e^{At}$ can be extended to the space $[\cD(A^*)]'$.
 \item
 $L$ is continuous from $L_p(0,T;X)$ to $C(0,T; [\cD(A^*)]')$ for every $p \in [1,\infty]$.
\item If $u \in C^1(0,T;X)$, then $Lu \in C(0,T;X)$.
 \item The condition
\begin{enumerate}
\item[(C3)] There exists a constant $C_T>0$ such that
\begin{equation*}
\int_0^T ||B^*e^{A^*t}x^*||^2_{\cU^*} dt \le C_T||x^*||^2_{X^*},
~~\forall\, x^*\in \cD(A^*)\subset X^*,
\end{equation*}
\end{enumerate}
  is equivalent to the regularity property
 $$L:~L_2(0,T;\cU) \to C(0,T;X)~~ \text{is continuous,}
 $$ i.e., there exists a constant $k_T>0$ such that
 \begin{equation}\label{conv-reg}
 ||Lu||_{C(0,T;X)} \le k_T ||u||_{L_2(0,T;\cU)}.
 \end{equation}
 \item Lastly, assume additionally that $A$ generates a strongly continuous group $e^{At}$ (e.g., if $A$ is
 skew-adjoint) and suppose that $L:L_2(0,T;\cU) \to L_2(0,T;X)$ is continuous. Then $(C3)$ is satisfied and thus we have the estimate in \eqref{conv-reg} in this case.
\end{enumerate}
\end{theorem}

\subsection{Application of the Abstract Scheme}
 In what follows, we adhere to the approach taken in \cite{supersonic}. We provide the theorems for self-containedness of the exposition, but for many of the proofs below we provide only brief comments. For complete details, see the manuscript \cite{supersonic}.
We will require the auxiliary space: $$
Z \equiv \left\{y =\begin{pmatrix} \phi \\ \psi \\ u \\ v \end{pmatrix}\in Y :
-U\pd_x \phi+\psi \in H^{1}(\R^3_+) \right\}
$$
endowed with the norm
\[
\|y\|_Z=\|y\|_Y+\|-U\pd_x \phi+\psi\|_{H^{1}(\R^3_+)}.
\]
One can see that $Z$ is dense in $Y$.
We also note that by the definition of  $\cD(\bA)$, we have $\cD(\bA) \subset Z$
and thus  $Z' \subset [\cD(\bA)]'$ continuously.
\begin{lemma}\label{le:P-op}
The operator $\bP_{\#}$ given by \eqref{op-P}
 is a bounded linear mapping  from $H^2_0(\Om)$ into $Z'$.
 Moreover,  the following estimates are in force:
 \begin{equation}\label{y-to-z1}
   ||\bP_{\#} [u]||_{Z'}\le C_U \|u\|_{H^2(\Om)}, ~~~\forall\, u\in H^2_0(\Om),
 \end{equation}
 and also (with  $\la\in\R$, $\la\neq 0$):
  \begin{equation}\label{y-to-z2}
   ||(\la-\bA)^{-1}\bP_{\#}[ u]||_{Y}\le C_{U,\la} \|u\|_{H^2(\Om)}, ~~~\forall\, u\in H^2_0(\Om).
 \end{equation}
 In the latter case we understand
$(\la-\bA)^{-1}\,  : [\cD(\bA)]'\mapsto Y$  as
the inverse to  the operator $\la-\bA$ which is extended to a mapping  from
$Y$ to  $[\cD(\bA)]'$. We also have that \eqref{y-to-z1} and \eqref{y-to-z2}
imply that $\bP$ maps $Y$ into $Z'$ and
 \begin{equation*}
   ||\bP y ||_{Z'}\le C \|y\|_{Y}  ~~\mbox{and}~~
||(\la-\bA)^{-1}\bP y||_{Y}\le C \|y\|_{Y},~~ \forall\, y\in Y.
 \end{equation*}
\end{lemma}
\begin{proof}
 For $u\in H^2_0(\Om)$ and $\overline{y}=(\overline{\phi},\overline{\psi};\overline{u},\overline{v}) \in Z$,
 from \eqref{p-y-z} we
have
\[
|(\bP_{\#}[ u], \overline y)_Y| =~U |(\bA_0N\partial_x  u, \overline \psi)|
= U |<\partial_x u, \gamma[\overline\psi]>|.
\]
Since $\ga[\overline\psi] = \ga[-U\pd_x \overline\phi+\overline\psi]+
U\pd_x\ga[\overline\phi]$, we have from the trace
theorem  that
\begin{equation*}
    ||\gamma[\overline\psi]||_{H^{-1/2}(\R^2)}\le C \left[||(-U\pd_x \overline\phi+\overline\psi)||_{H^{1}(\R_+^3)} +\|\overline
\phi||_{H^{1}(\R^3_+)}\right]\le C\|\overline y\|_{Z} .
\end{equation*}
Therefore
\[
|(\bP_{\#} [u], \overline y)_Y|\le
 C(U)  ||\pd_x u||_{H^{1/2}(\Omega)}||\gamma[\overline\psi]||_{H^{-1/2}(\R^2)}
\le  ~ C(U)  || u ||_{H^{2}(\Omega)} \| \overline y\|_{Z}.
\]
Thus
\[
||\bP_{\#} [u]||_{Z'} = \sup\big\{
|(\bP_{\#} [u], \overline y)_Y|\, : \overline y \in Z,~||\overline y||_Z =1\big\}
\le  C(U)  || u ||_{H^{2}(\Omega)},
\]
which implies \eqref{y-to-z1}.
 The relation in \eqref{y-to-z2} follows from  \eqref{y-to-z1} and
the boundedness of the operator
$(\la-\bA)^{-1}\, : [\cD(\bA)]'\mapsto Y$.
\end{proof}
We may now consider mild solutions to the problem given in (\ref{inhomcauchy}).
Applying general results on $C_0$-semigroups (see \cite{Pazy}) we arrive at the following assertion.
\begin{proposition}\label{pr:mild}
Let $\overline{u} \in C^1([0,T];H^{2}_0(\Omega))$ and $y_0\in Y$. Then $y(t)$
given by \eqref{solution1} belongs to  $C([0,T];Y)$ and
is a strong solution to \eqref{inhomcauchy} in $[\cD(\bA)]'$, i.e.
in addition we have that
\[
y\in  C^1((0,T);[\cD(\bA)]')
\]
and  \eqref{inhomcauchy} holds in $[\cD(\bA)]'$ for each $t\in (0,T)$.
\end{proposition}

Proposition~\ref{pr:mild} implies that $y(t)$ satisfies
the variational relation
\begin{equation}\label{weak-eq}
    \pd_t(y(t),h)_Y=-(y(t), \bA h)_Y+ (\bP_{\#}[\overline{u}(t)],h)_Y,~~~\forall\, h\in \cD(\bA).
\end{equation}
In our context in application of Theorem~\ref{dualityequiv} we have that $X=Y$,
where $Y$ is Hilbert space given by
\eqref{space-Y}, $\cU=H^2_0(\Om)$ and $B=\bP_\#$ is defined by \eqref{op-P}
as an operator from  $\cU=H^2_0(\Om)$ into $Z'$ (see Lemma~\ref{le:P-op}).
One can see from \eqref{p-y-z} that the adjoint operator $\bP_\#^*\, : Z\mapsto H^{-2}(\Om)$ is given by
\[
\bP^*_{\#}z=U\big[\partial_x N^*(\bA^*_0+I) \overline \psi\big]\big|_\Om = U\partial_x \gamma[\overline \psi]\big|_\Om,
 ~~\mbox{for}~~ \ds z = \begin{pmatrix} \overline \phi \\ \overline \psi \\ \overline u \\ \overline v\end{pmatrix}\in Z.
 \]
Condition $(C3)$ in Theorem~\ref{dualityequiv} (4) is then paraphrased by writing
\begin{equation*}
y\mapsto
 \bP^*_\#e^{\bA^*(T-t)}y\equiv \bP^*_\#e^{\bA t}y_T : ~\text{continuous}~ Y \to L_2\big(0,T;H^{-2}(\Omega)\big),
\end{equation*}
where $y_T=e^{\bA^*T}y=e^{-\bA T}y$ (we use here the fact that $\bA$ is a skew-adjoint m-dissipative operator).
\par
Let us denote by
$e^{\bA t}y_T\equiv  w_T(t) =(\phi(t),\psi(t); u(t), v(t))$
the solution of the linear problem in \eqref{eq-trankt-0} with initial data $y_T$.
Then from the trace estimate in (\ref{trace-reg-est-M0}) we have that
 \begin{align*}
||\bP^*_\#e^{\bA t}y_T||^2_{L_2(0,T:H^{-2}(\Omega))} = &~ U\int_0^T ||\partial_x\ga[\psi(t)]||^2_{H^{-2}(\Omega)}dt \\ \nonumber
\le
& C \int_0^T ||\gamma[\psi(t)]||^2_{H^{-1/2-\epsilon}(\Omega)}  dt \le C_{T} \left(E_{fl}(0)+\int_0^T ||v(t)||^2_{L_2(\Omega)}\right)dt \\ \nonumber
\le &  C_{T}\| y_T\|_Y=C_{T}\| e^{-\bA T}y\|_Y  =C_{T}\|y\|_Y.
\end{align*}
In the last equality we also use that
 $ e^{\bA t}$ is a $C_0$-group of isometries.

Now we are fully in a position to use Theorem \ref{dualityequiv}
which leads to the following assertion.

\begin{theorem}[{\bf $L$ Regularity}]
Let $T>0$ be fixed, $y_0 \in Y$ and $\overline u \in C([0,T]; H_0^{2}(\Omega))$.
Then  the mild solution
 \begin{equation*}
 \ds y(t)=e^{\bA t}y_0+L[\overline u](t)\equiv
 e^{\bA t}y_0+\int_0^te^{\bA (t-s)}\bP_{\#}[\overline u(s)] ds
 \end{equation*}
 to problem \eqref{inhomcauchy} in $[\cD(\bA)]'$ belongs to the class $C([0,T]; Y)$
 and enjoys the estimate
\begin{align}\label{followest}
\max_{\tau \in [0,t]} ||y(\tau)||_Y \le&  ||y_0||_Y + k_T ||\overline u||_{L_2(0,t;H^{2}_0(\Omega))},~~~\forall\, t\in [0,T].
\end{align}
\end{theorem}
\begin{remark}
We emphasize that the perturbation $\bP$ acting outside of $Y$ is regularized when incorporated into the operator $L$ defined above; namely, the variation of parameters operator $L$ is a priori only continuous from $L_2(0,T;\cU)$ to $C(0,T;[\cD(\bA ^*)]')$. However, we have shown that the additional ``hidden'' regularity of the trace of $\psi$ for solutions to (\ref{flow}) with the boundary conditions in \eqref{KJcondition} allows us to bootstrap $L$ to be continuous from $L_2(0,T;\cU)$ to $C(0,T;Y)$ (with corresponding estimate) via Theorem \ref{dualityequiv}. This result  justifies \textit{formal} energy methods on the equation (\ref{inhomcauchy}) in order to produce a fixed point argument.\end{remark}

\subsection{Construction of a Generator}
Let  $\bX_t = C\big((0,t];Y\big).$
Now, take $\overline{y}=(\overline{\phi},\overline{\psi};\overline{u},\overline{v}) \in \bX_t$ and $y_0 \in Y$, and introduce the map $\cF: \overline{y} \to y$ given by
\begin{equation*}
y(t) = e^{\bA t}y_0+L[\overline u](t),
\end{equation*}
i.e. $y$ solves \begin{equation*}
y_t=\bA y+\bP_\# \overline u,
~~ y(0)=y_0,
\end{equation*}
in the generalized sense, where $\bP_\#$ is defined in \eqref{op-P}.
 It follows from (\ref{followest})
 that for $\overline y_1, ~\overline y_2 \in \bX_t$
 \begin{align*}
 \|\cF \overline y_1 - \cF \overline y_1\|_{\bX_t}\le & ~k_T ||\overline u_1 - \overline u_2||_{L_2(0,t;H^{2}_0(\Omega))} \\
 \le & ~k_T \sqrt{t}\max_{\tau \in [0,t]}|| \overline u_1 - \overline u_2||_{H^2(\Omega)}
 \le k_T \sqrt{t} ||\overline y_1 - \overline y_2 ||_{\bX_t}.
 \end{align*}
Hence there is $0<t_*<T$ and $q<1$ such that
\[
\|\cF \overline y_1 - \cF \overline y_2\|_{\bX_t}\le q \| \overline y_1 -  \overline y_2\|_{\bX_t}
\]
for every $t\in (0,t_*]$.
This implies that  on the interval $[0,t_*]$ the problem
\begin{equation*}
y_t=\bA y+\bP y, ~~t>0,~
~~ y(0)=y_0,
\end{equation*}
has a local in time unique (mild) solution defined now in $Y$. This above local solution
can be extended to a global solution in finitely many steps by linearity.
Thus there exists a unique function
$y=(\phi,\psi;u,v)\in C\big(\R_+;Y\big)$ such that
\begin{equation}\label{eq-fin}
y(t)=e^{\bA t}y_0+\int_0^te^{\bA(t-s)}\bP [y(s)]ds~~\mbox{ in }~ Y
~~\mbox{for all }~t>0.
\end{equation}
It also follows from the analysis above that
\[
\|y(t)\|_Y\le C_T \|y_0\|_Y,~~~ t\in [0,T],~~\forall\, T>0.
\]
Thus the problem \eqref{eq-fin} generates strongly continuous semigroup
$\widehat{T}(t)$ in $Y$.
Additionally, due to \eqref{weak-eq} we have
\begin{equation*}
    (y(t),h)_Y=(y_0,h)_Y+ \int_0^t\left[
    -(y(\tau), \bA h)_Y+ (\bP[y(\tau)],h)_Y
    \right]d\tau
    ,~~~\forall\, h\in \cD(\bA),
    ~~ t>0.
\end{equation*}
Using the same idea as presented in \cite{jadea12, supersonic} (which relies on Ball's Theorem \cite{ball} and \cite{desh}), we conclude that
the generator $\widehat{\bA}$ of $\widehat{T}(t)$  has the form
\[
\widehat{\bA}z=\bA z+\bP z,~~z\in\cD(\widehat{\bA})=\left\{z\in Y\,:\; \bA z+\bP z\in Y\right\}
\]
(we note that  the sum $\bA z+\bP z$ is well-defined as an element
in $[\cD(\bA)]'$ for every $z\in Y$). Hence, the semigroup $e^{\widehat \bA t}y_0$ is a generalized solution for $y_0 \in Y$ (resp. a classical solution for $y_0 \in \cD(\widehat \bA)$) to (\ref{sys1*}) on $[0,T]$ for all $T>0$.
\begin{equation}\label{dom-bA-n}
\cD(\bA + \bP) \equiv \left\{ y \in Y\; \left| \begin{array}{l}
-U \Dx \phi + \psi \in H^1(\R^3_+),~\\ -U \Dx \psi  -\bA_0 (\phi - N (v + U \Dx u ))+N(v+U\Dx u) \in L_2(\R^3_+) \\
v \in \cD(\cA^{1/2} )= H_0^2(\Omega),~
-\cA u + N^* (\bA_0^*+I) \psi \in L_2(\Omega) \end{array} \right. \right\}
\end{equation}
Now we can conclude the proof of generation for the full dynamics of \eqref{flowplate} (represented by $\bA+\bP$); indeed, the
function $y(t)$ is a generalized solution corresponding to the generator $\bA + \bP$ with the  domain defined in (\ref{dom-bA-n}). 

\section{Trace Regularity}\label{tracereg}
The ``hidden" trace regularity of the term $\psi$ coming from the flow equation \eqref{flow} is {\em critical} to the arguments above. In this section we analyze this problem in the dual (Fourier-Laplace) domain and relate it to a certain class of integral transforms reminiscent of the finite Hilbert transform. In the case of two dimensions, we reduce the trace regularity to an hypothesis about the invertibility of Hilbert-like transforms on bounded domains. Additionally, we demonstrate the necessary trace regularity by performing microlocal analysis on a pseduodifferential operator corresponding to the flow problem in one dimension. 

We are interested in the trace regularity of the following flow problem in $\realsthree_+$: \begin{equation}\label{flow1}\begin{cases}
(\partial_t+U\partial_x)^2\phi=\Delta \phi & \text { in } \realsthree_+ \times (0,T),\\
\phi(0)=\phi_0;~~\phi_t(0)=\phi_1,\\
\Dz \phi = d(\xb,t)& \text{ on } \Omega \times (0,T)\\
\gamma[\phi_t+U\phi_x]=\gamma[\psi]=0, ~&\xb \in \realstwo \backslash \Omega,
\end{cases}
\end{equation}
with $0 \le U <1$,  $d(\xb,t) \in L_2(0,T;L_2(\Omega))$ is the ``downwash" generated on the structure, and $\gamma [\psi] $ is the aeroelastic potential. 
\begin{remark}
An related analysis of  trace regularity for this equation was carried out in the case of $L_2^{loc}(\realstwo)$ Neumann data in \cite{supersonic}. Here, the mixed boundary conditions present a formidable challenge in the microlocal analysis.
\end{remark}
We recall  {\bf Condition \ref{le:FTR0}}:

{\em If $\phi(\xb,t)$  satisfies \eqref{flow1} then  for $\epsilon > 0 $  $$\partial_t\gamma[\phi]
 \in L_2(0,T;H^{-1/2-\epsilon}(\realstwo))~~~~\forall\, T>0.
$$
Moreover,  with $\psi = \phi_t + U \phi_x $ we have
\begin{equation}\label{trace-reg-est-M}
\int_0^T\|\gamma[\psi](t)\|^2_{H^{-1/2 -\epsilon} (\realstwo)}dt\le C_T\left(
E_{fl}(0)+
 \int_0^T\| \Dn \phi(t) \|_{\Omega}^2dt\right)
\end{equation}}
\begin{remark} 
Note that  $\partial_x\gamma[\phi] \in C(0, T; H^{-1/2}(\realstwo) $ follows from a priori interior regularity of $\phi$ and the fact that $x$ direction is tangential to the boundary. Thus, the only real requirement is  the trace  regularity of  the time derivative of $\phi$. 
\end{remark}
\begin{remark}
We recall that related regularity of  the trace to $\phi_t$ was derived  in \cite{supersonic} which deals with the Neumann boundary data. In fact, it was shown there that 
in the case of Neumann boundary conditions with both subsonic and supersonic velocities  $U$, 
 one obtains the regularity as in (\ref{trace-reg-est-M} ) with $\epsilon =0 $. 
 Though these are related regularity results, the techniques of obtaining them are very different. In the Neumann case we perform microlocal analysis by microlocalizing hyperbolic and elliptic sectors. In the present case, the regularity phenomenon has to do with spectral analysis of finite Hilbert transforms. 
 \end{remark}

In line with the similar analyses in \cite{miyatake1973,sakamoto}, we take zero initial data (the principle of superposition may then be applied). We consider the Fourier-Laplace transform of the original linear equation (formally, sending $(x,y) \to i(\eta_x,\eta_y)$ and $t \to \tau=(\alpha +i\beta$)), resulting in: 

\begin{equation} (\tau+iU\eta_x)^2\hat{\phi}+|\eta|^2\hat{\phi}=( \tau^2  + 2U i \eta_x \tau  + (1- U^2) \eta_x^2 + \eta_y^2 ) \hat{\phi} = \Dz^2 \hat{\phi}.\end{equation}
Let us denote 
\begin{equation}\label{D} D (\eta, \tau) \equiv (\tau^2  + 2U i \eta_x \tau  + (1- U^2) \eta_x^2 + \eta_y^2). \end{equation}   
Then, we must solve (in $z$)
$$\partial_z^2 \hat{\phi} = D (\tau, \eta)  \hat{\phi};$$ doing so, we
obtain 
 $ \hat{\phi}(\tau,\eta,z) = \hat{\phi} (\tau, \eta, 0) e^{-z \sqrt{D}}. $ 

 Our next step is to relate the boundary conditions to the the normal derivative and the acceleration potential $\psi$: 
 $$ \hat{\psi}(\tau,\eta, z) = (\tau + i  U \eta_x )  \hat{\phi} (\tau, \eta, 0) e^{-z \sqrt{D}} $$
 $$ \Dz \hat{\phi} = - \sqrt{D}  \hat{\phi} (\tau, \eta, 0)$$ 
 Hence, on $\partial \realsthree_+$  we have the relation
\begin{equation}\label{microrelation} \Dz \hat{\phi} = \frac{-\sqrt{D}}{\tau + iU \eta_x } \hat{\psi} \equiv m(\tau, \eta) \hat{\psi}. \end{equation}
This relation is supplemented with the information 
\begin{equation}\label{microBC} \begin{cases}
\Dz \phi = d& ~\text{ in }~  \Omega \\
{\psi}(t, x, y, 0)  =0&  ~\text{ outside }~\Omega. \end{cases}
\end{equation}
\begin{remark}
The boundary conditions above in \eqref{microBC} are the {\em key} feature which distinguish this analysis of trace regularity from that in \cite{supersonic}. \end{remark}

In what follows, we accommodate the mixed nature of the boundary conditions below by decomposing the symbol in \eqref{microrelation} into the product of two symbols which can be separately analyzed.
Let $\cT$ denote the pseudodifferential operator corresponding to the multiplier $m (\eta, \tau)$, and
let $P_{\Omega}: \realstwo \to \Omega $ denote projection on $\Omega$. 
Then for any $\psi$ with support contained in $\Omega$, the integral equation in \eqref{microrelation}--\eqref{microBC}  can be recast as 
\begin{equation}\label{P}
d =P_{\Omega}\cT\psi, ~~\xb \in \Omega, ~~t > 0, 
\end{equation}
noting that $\cT$ acts on a ``truncated" function $\psi$, so that 
the PDO operator 
$P_{\Omega} \cT $ can be viewed  as a PDO operator on $\Omega \times (0, \infty)$. 
\begin{remark}
We can formally construct $\cT$ as follows:
for any  $\psi$ with $\psi =0 $ outside $\Omega$, we construct Laplace (time) Fourier (space) transform of $\psi$ which
we denote $\hat{\psi} (\tau, \eta) $. Then the operator $\cT $ is the PDO operator 
such that $$(\widehat{\cT \phi})(\tau, \eta)  = m(\tau, \eta) \hat{\phi} (\tau, \eta).$$ Taking the inverse of the Fourier-Laplace transform, and applying projection on $\Omega$ gives the appropriate statement Possio type equation. \end{remark}

Equation (\ref{P}) is an abstract version of a Possio integral equation \cite{bal0,bal4} in two dimensions. We note here that by that we have already solved for $\psi$ given $d$; indeed, in analyzing the operator $\bA_0$ above, we have deduced solvability of the system in \eqref{flow1} for $(\phi,\psi) \in H^1(\realsthree_+)\times L_2(\realsthree_+)$. We are therefore interested in characterizing the solution: {\em given $h$ on $\Omega$ with some regularity, find the corresponding regularity of  aeroelastic potential $ \gamma[\psi]$ on $\Omega$}. 
 \begin{remark}
The connection between integral equations appearing in the study of aeroelasticity and invertibility of finite Hilbert fransforms has been known for many years and dates back to Tricomi and his airfoil equation \cite{tricomi}. This approach has been critically used in \cite{bal0,bal4} where the analysis is centered on solvability of integral equations connecting the downwash with the aeroelastic potential. In these works the author performs an analysis on a one dimensional flow-beam system and utilizes a similar Fourier-Laplace approach and corresponding $L_p$ theory. However, the solution to the system (the Possio integral equation) is {\em constructed} via a Fourier-Laplace approach. We follow the same conceptual route with different technical tools. Our approach is based on microlocal analysis---rather than explicit solvers of integral equations arising in a purely one dimensional setting. Though our final estimate depends on an assumption (which we demonstrate only for the one dimensional case) we believe that the microlocal approach provides new ground for extending the flow-structure analysis to the multidimensional settings.
 \end{remark}
 The relationship between $\psi$ and $d$ depends on the properties of the (Mikhlin) multiplier $m(\tau, \eta )$. In fact, as we will show below, $\cT$ is related to the finite Hilbert transform in the $x,y$ variable. We later make this notion precise, but we note at this stage that the integral transform which arises is {\em not standard in the two dimensional scenario}; however, when $\Omega$ is reduced to an interval, the multiplier above contains the symbol of the classical finite Hilbert transform. Define an operator $\mathscr{H}$ on $L_2(\Omega)$ whose symbol is given by $symb(\mathscr H)= \dfrac{-i|\eta_x|}{\eta_x}\equiv j(\eta_x)$, where $\eta_x \leftrightarrow \frac{1}{i}\partial_x$ in the Fourier correspondence. Additionally, let $P_{\Omega}:\reals^1 \to \Omega$ also represent the associated projection into $L_p(\Omega)$ and $E_{\Omega}$ be the extension operator (by zero) from $L_p(\Omega)$ into  $L_p(\realstwo)$.  Then define the operator $\cH_f$ as follows:
 \begin{equation}\label{finhilb}
 \mathscr{H}_f\equiv P_{\Omega}\mathscr{H}E_{\Omega}:L_p(\Omega) \to L_p(\Omega).
 \end{equation}
When $\Omega$ is a disjoint union of open intervals, this operator corresponds to the so called finite Hilbert transform (see Section \ref{hilbert}).

\subsection{Characterization of the Multiplier $m(\tau,\eta)$} 
Recall $D(\tau, \eta) \equiv  (\tau+iU\eta_x)^2+|\eta|^2$, with $\tau = \alpha + i \beta$ ($\alpha>0$ fixed) and $|\eta|^2 = \eta_x^2+\eta_y^2,$ so  
\begin{equation} D(\tau, \eta) = \alpha ^2 - \beta^2 + (1-U^2) \eta_x^2 + \eta_y^2  - 2 U \beta \eta_x 
+ 2 i \alpha (U \eta_x + \beta ).  \end{equation}
 After some calculation and  recalling $ m = - \frac{\sqrt{D}}{ \tau + i U \eta_x}$, we have a representation for $m$ (as given in \eqref{microrelation}):
 \begin{align}\label{symb}  m(\tau, \eta) = j(\eta_x) S(\tau, \eta) 
 \end{align}
 The auxiliary symbol $S$ is given by
 \begin{align}S(\tau,\eta)  \equiv& \dfrac{-i \sqrt{(\frac{\tau}{\eta_x}+iU)^2+1+\left(\frac{\eta_y}{\eta_x}\right)^2}}{\frac{\tau}{\eta_x}+iU},\end{align} and $$j(\eta_x)=-i\dfrac{|\eta_x|}{\eta_x}.$$
 In the analysis of the symbol $m(\tau,\eta)$, we must characterize the singular integral operator corresponding to the symbol $j(\eta_x)$. To the knowledge of the authors, this does not correspond to a well-known problem in dimensions higher than one. For this reason we shall specify the analysis to the case when $\Omega$ is an interval. 
We recall that in one dimension $j(\eta) =\ds \frac{-i|\eta|}{\eta} $ is the symbol associated to the Hilbert transform. 
  
\noindent {\bf \em Assumption:} In what follows, we specialize to the case when $\Omega\subset \reals$ is the interval $(-1,1)$ and the flow resides in the plane $\{(x,y,z): y=0, z\ge0\},$ and demonstrate the necessary estimates (which then becomes a prototype for the two dimensional problem). As such, we now suppress the span variable by restricting to $y=0$, and employ the notation  $\eta = \eta_x $. We note that this is similar to the approach taken to this problem in \cite{bal0,bal4} which reduce  to a two dimensional flow interacting with the one dimensional structure. In this case, we have that the operator $\cH _f:L_p(-1,1) \to L_p(-1,1)$ is the finite Hilbert transform \cite{FH1,FH,FH2} and the Appendix. We will use the fact that the finite Hilbert transform is invertible on $L_{2^-}(-1,1)$. 

\begin{theorem}\label{1dtracereg}
Assume $\Omega=(-1,1)$ (suppressing the span variable $y$). Then the flow trace regularity in Condition \ref{le:FTR0} holds for $\phi$.
\end{theorem}
\begin{proof}
In line with hyperbolic character of the symbol, it is also convenient to introduce the variable 
$ z_U\equiv \beta+U\eta$.
With this notation we simplify (choosing the positive square root)
\begin{equation}\label{S} S (z_U, \eta) = -i\eta\dfrac{\sqrt{\frac{1}{\eta^2}\left[(\alpha+iz_U)^2+\eta^2\right]}}{\alpha+iz_U}\end{equation}
We are interested in the invertibility (and regularity properties of the inverse) of the operator associated to the multiplier $$m(z_U,\eta)=\dfrac{-i|\eta|}{\eta}S(z_U,\eta).$$ Properties of the inverse finite Hilbert transform are well-understood (see \cite{FH,FH1,FH2} and the Appendix). 

\begin{lemma}\label{mikhlin} The multiplier $S(z_U,\eta)(1+\eta^2)^{1/4}$ is bounded from below for all $\eta,z_U \in \reals$. \end{lemma}
 \begin{proof} To unify our approach with that in \cite{supersonic}, we show that  $M(z_U,\eta)\equiv [S(z_U,\eta)[1+\eta^2]^{1/4}]^{-1}$. 
  is bounded from above for all $\eta,~z_U \in \reals$. 
  
 \begin{align}
|M(z_U,\eta)|=&\dfrac{1}{(1+\eta^2)^{1/4}}\left|\dfrac{\alpha+iz_U}{\sqrt{(\alpha^2-z_U^2+\eta^2)+2i\alpha z_U}} \right|\\\label{7}
=&\dfrac{1}{(1+\eta^2)^{1/4}}\dfrac{\sqrt{\alpha^2+z_U^2}}{\left((\alpha^2-z_U^2+\eta^2)^2+4\alpha^2z_U^2\right)^{1/4}}\\\label{8}
=&\dfrac{1}{(1+\eta^2)^{1/4}}\dfrac{\sqrt{\alpha^2+z_U^2}}{\left((\eta^2-z_U^2)^2+2\alpha^2\eta^2+\alpha^4+2\alpha^2z_U^2\right)^{1/4}}
\end{align} We analyze the symbol via cases: $$\begin{cases}\text{A:} & 0 \le |z_U| < \frac{1}{2}|\eta| \\ \text{B:} &\frac{1}{2}|\eta| \le |z_U| \le 2|\eta| \\ \text{C:} & |z_U|>2|\eta| \end{cases}.$$ First, in Case A, we note that the quantity $\eta^2-z_U^2>0$, and hence, directly from \eqref{8} (discarding terms) we have
\begin{equation}
|M(z_U,\eta)| \le \dfrac{1}{(1+\eta^2)^{1/4}}\dfrac{\sqrt{\alpha^2+\frac{1}{4}\eta^2}}{\left(\frac{9}{16}\eta^4+\alpha^4\right)^{1/4}}<\left[ \dfrac{2}{1+\eta^2}\right]^{1/4},
\end{equation}  for all $\eta$ and $z_U$.

In Case C, we utilize the characterizing bound in the principal term to arrive at
\begin{equation}
|M(z_U,\eta)| \le \dfrac{1}{(1+\eta^2)^{1/4}}\dfrac{\sqrt{\alpha^2+z_U^2}}{\left(\frac{1}{16}z_U^4+\alpha^4\right)^{1/4}}<\left[ \dfrac{81}{1+\eta^2}\right]^{1/4},
\end{equation} for all $\eta$ and $z_U$.

In Case B the principal portion of the denominator can degenerate, indicating the need for the $[1+\eta^2]^{1/4}$ term incorporated into the symbol $M$:
\begin{align}
|M(z_U,\eta)|\le &\dfrac{1}{(1+\eta^2)^{1/4}} \dfrac{\sqrt{\alpha^2+4\eta^2}}{\left[\frac{5}{2}\alpha^2\eta^2+\alpha^4\right]^{1/4}}\le [5+8\alpha^2]^{1/4}.
\end{align}
Hence, there exists some $C>0$ such that $M(z_U,\eta)$ is bounded as follows
$$|M(z_U,\eta)| \le C\left[1+\alpha^2+\frac{1}{1+\eta^2}\right]^{1/4}.$$ Returning to the symbol $S(\tau,\eta)$, we have that $$[1+\eta^2]^{1/4}|S(\tau,\eta)| >c\left[\dfrac{1+\eta^2}{(1+\eta^2)(1+\alpha^2)+1} \right]^{1/4}.$$\end{proof}

 Hence, we have by the analysis above (taking the pseudodifferential operator $\mathscr{S}$ to correspond to the auxiliary symbol $S(\beta,\eta)$): \begin{corollary}\label{Sinv}
 $$ \mathscr{S}^{-1}: L_2(0,T; H^{-\epsilon}  (\reals) )   \rightarrow L_2(0,T; H^{-1/2 - \epsilon}(\reals))  , \epsilon > 0 $$ is bounded.
  \end{corollary} \noindent  (We note the familiar loss of $1/2 $ derivative \cite{supersonic}.)
 At this stage that the integral transform which arises corresponding to the so called  abstract version of Possio equation \cite{bal0,bal4}  
$$d=P_{\Omega}\mathscr T \psi,~\text{on}~ \Omega , t > 0 ,$$
We recall that    $\psi =0$ off $\Omega$, so $E_{\Omega}\psi $ coincides in this case with $\psi$.
Given $d \in C(0, T; L_2(\Omega))$ our task is to infer the regularity of $ \psi $. 
This precise formulation is the one given in \cite{bal0}, where the connection with finite Hilbert transform has been used.

Relating the Hilbert transform $\mathscr H$ to the finite Hilbert transform $\mathscr H_f$ described above, one obtains 
\begin{align}\label{possio}
d =& P_{\Omega}\cT \psi 
=  P_{\Omega} \mathscr{H} \mathscr{S}  \psi\\
=& P_{\Omega} \mathscr{H} E_{\Omega} P_{\Omega } \mathscr{S} \psi \nonumber +P_{\Omega} \mathscr{H} [ I - E_{\Omega} P_{\Omega } ]  \mathscr{S}  \psi \nonumber \\
 =& \mathscr{H}_f P_{\Omega} \mathscr{S} \psi + V  \mathscr{S} \psi
 \end{align}

 where  $V\mathscr{S} \psi =[ I - E_{\Omega} P_{\Omega} ] \mathscr{S} \psi$. 
 Since the singular support of $V\mathscr{S} \psi $ is empty,
 by the pseudolocal properties of pseudodifferential operators,  the operator $V$ is ``smooth " and compact (see  \cite{bal0,bal4} for detailed calculations with a similar decomposition). The operator $\mathscr{H}_f P_{\Omega} + V $  is therefore a compact perturbation of Finite Hilbert Transform, injective \cite{bal0}, hence  invertible on  $L_2(0, T; L_p(\Omega))$ with $ p < 2$ (by the invertibility properties of the finite Hilbert transform---see the Appendix). 
  
With $d \in L_2(0,T; L_2(\Omega))$, inverting,  we have $\mathscr{S} \psi \in L_2(0, T; L_{p}(\reals))$ with $p < 2$, which yields via the Sobolev embedding $$H^{n\delta/(2+\delta)}(\reals^n) \hookrightarrow L_{2-\delta}(\reals^n)$$ that$ \mathscr{S} \psi \in L_2(0, T;  
  H^{-\epsilon} (\reals) )  $ for every $\epsilon> 0$ by taking a suitable $ p < 2 $. Thus, utilizing Corollary \ref{Sinv},  $\psi \in L_2(0,T; H^{-1/2-\epsilon}(\Omega))$, as desired, which concludes the proof of Theorem \ref{1dtracereg}. 
 \end{proof}
 
\begin{remark}
Noting that the Sobolev embedding we used is not critically affected by  the dimension of $\Omega$, we have thus provided the main motivation for Condition \ref{le:FTR0}. What is missing however, is the analog of theory of finite Hilbert Transforms carried at the two dimensional level.

\end{remark}
\begin{remark}\label{Bal1}
We note that the same result (essentially) follows from the analysis in \cite{bal0,bal4}, where the author proves that 
aeroelastic potential $ \psi \in L_2(0, T , L_q(\Omega) ) $ for $ q < 4/3 $ . Since  for $p > 4 $ there exists $\epsilon > 0 $ 
such that $$H^{1/2 + \epsilon} (\Omega) 
\subset L_p(\Omega) , ~~p > 4 , ~~\text{dim}~~\Omega \leq 2,$$  and one then obtains  that $L_q(\Omega) \subset H^{-1/2 - \epsilon}(\Omega) $ with $q < 4/3 $. \end{remark}
\begin{remark}
The loss of $1/2$ derivative in the characteristic region was already observed and used in the analysis of regularity of 
the aeroelastic potential for the Neumann problem with supersonic velocities $U$ \cite{supersonic}. 
However, in the case of K-J conditions there is an additional loss, due to the necessity of inverting finite Hilbert transform. 
\end{remark}

We now bring our attention back to the case when $\Omega$ is a two dimensional domain. The analysis above can be performed analogously for $S(\tau,\eta)$ when $\eta=(\eta_x,\eta_y)$; however, as we see from \eqref{symb}, we must have invertibility estimates on the operator associated to the symbol $\dfrac{i|\eta|}{\eta}$ in a vectorial setting for $\eta$ and  for a two dimensional domain. This corresponds to a  finite Riesz  transform {\em in the $x$ direction.} The trace regularity analysis (as done above) will depend on solutions to singular integral equations which to the authors' knowledge are not readily available.  Such results, should they exist, will depend highly on the geometry of domains. As such, the corresponding trace regularity result for $\Omega$ a two dimensional domain {\em as in \eqref{trace-reg-est-M0}} will also depend on the geometry; ultimately, this assumption will be verified if properties of the finite Hilbert transform carry over to the higher dimensional transforms mentioned previously. We now state this as a lemma:
\begin{lemma}
Assume that the operator $\mathscr{H}_f\equiv P_{\Omega}\mathscr{H}E_{\Omega}:L_p(\Omega) \to L_p(\Omega)$ is continuously invertible for $p \in (1,2)$. Then the trace regularity condition Condition \ref{le:FTR0} holds.
\end{lemma}
As discussed above and in the Appendix, the hypothesis of this lemma is a generalization of the invertibility properties of the finite Hilbert transform which were critical in the proof of Theorem \ref{1dtracereg} above.

\section{Physical Considerations}\label{physicals}
\subsection{Nonlinearities}
We briefly mention the three principal nonlinearities associated to the flow-plate interaction model presented above, appearing in the plate equation, i.e. $u_{tt}+\Delta^2u+f(u) = p(\xb,t).$
In the standard analyses \cite{springer,webster,supersonic} we consider a general situation that
covers nonlinear (cubic-type) force terms $f(u)$ resulting from  aeroelasticity  modeling \cite{bolotin,dowell,dowell1}.
These include:
\begin{enumerate}
  \item {\sl Kirchhoff model}: $u\mapsto f(u)$ is the Nemytski operator
with a function $f\in {\rm Lip_{loc}}(\R)$ which fulfills the  condition
\begin{equation}
    \label{phi_condition}
    \underset{|s|\to\infty}{\liminf}{\frac{f(s)}{s}}>-\la_1,
\end{equation}  where $\la_1$ is the first eigenvalue of the
biharmonic operator with homogeneous Dirichlet boundary conditions.
  \item {\sl Von Karman model:} $f(u)=-[u, v(u)+F_0]$, where $F_0$ is a given function
  from $H^4(\Om)$ and
the von Karman bracket $[u,v]$  is given by
\begin{equation*}
[u,v] = \partial ^{2}_{x} u\cdot \partial ^{2}_y v +
\partial ^{2}_y u\cdot \partial ^{2}_{x} v -
2\cdot \partial ^{2}_{xy} u\cdot \partial ^{2}_{xy}
v,
\end{equation*} and
the Airy stress function $v(u) $ solves the following  elliptic
problem
\begin{equation}\label{airy-1}
\Delta^2 v(u)+[u,u] =0 ~~{\rm in}~~  \Omega,\quad \Dn v(u) = v(u) =0 ~~{\rm on}~~  \pd\Om.
\end{equation}
Von  Karman equations are well known in nonlinear elasticity and
constitute a basic model describing nonlinear oscillations of a
plate accounting for  large displacements, see \cite{karman}
and also \cite{springer,ciarlet}  and
references therein.
  \item {\sl Berger Model:} In this case the feedback force $f$ has the form
  $$
  f(u)=\left[ \kappa \int_\Om |\nabla u|^2 dx-\Gamma\right] \Delta u,
$$
where $\kappa>0$ and $\Gamma\in\R$ are parameters, for some details and  references see \cite{berger0} and also
\cite[Chap.4]{Chueshov}.
\end{enumerate}

The following proposition allows us to incorporate the above nonlinearities into our theory of well-posedness as in \cite{webster,jadea12,supersonic}. For this reason we do not elaborate further on the nonlinear analysis, as the key issues arising in the study of this model which stand apart from the aforementioned analyses occur in the {\em linear theory}.

\begin{proposition}\label{abstractnonlin}
For each of the nonlinearities $f$ above, we have that $f$ is locally Lipschitz from
 $H^2_0(\Omega)$ into $L_2(\Om)$ and there exists
$C^1$-functional $\Pi(u)$ on  $H^2_0(\Omega)$ such that
$f$ is a Fr\'echet derivative of $\Pi(u)$,
$f(u)=\Pi'(u)$. Moreover, $\Pi(u)$
 is  locally bounded on  $H^2_0(\Omega)$, for all $\delta$ sufficiently small
 there exists a $C_{\delta}\ge 0$ such that
\begin{equation}\label{potentialbound}
0 \le \delta \|\Delta u\|_\Om^2 +\Pi(u)+C_{\delta}\;,\quad \forall\, u\in  H^2_0(\Omega)\;.
\end{equation}
\end{proposition}

\subsection{Other Configurations with Free Boundary Conditions}\label{physical1}
The configuration considered in this treatment represents an attempt to understand PDE aspects of the dynamic, mixed K-J conditions. To do so, we have taken clamped plate boundary conditions. However, in recent discussions with E. Dowell (Duke) \cite{corres,K1,K2,K4}, the authors have come to understand that  the {\em free-clamped} configuration represents a model of great recent interest. In addition, it is perhaps the most mathematically interesting (and difficult) case corresponding to this class of flow-plate models. These configurations are extremely important in the 
modeling of airfoils and in the modeling of panels in which some component of the boundary is left free.
In addition to K-J condition, one must contend with the difficulties associated with the free plate boundary condition.
The applicability of K-J boundary condition is highly dependent upon the geometry of the plate in question.

 The configuration below represents an attempt to model oscillations of a plate which is {\em mostly free}. The dynamic nature of the flow conditions correspond to the fact that the interaction of the plate and flow is no longer static along the free edge(s), and in this case the implementation of the K-J condition is called for. This yields the following boundary conditions for the flow-plate system:
\begin{equation}\label{physical}\begin{cases}
u=\Dn u = 0,  &\text{ on }  \pd\Om_1\times (0,T) \\
\mathcal{B}_1 u =0,   \mathcal{B}_2 u =0,~ &\text{on}~ \pd\Om_2\times (0,T)\\
\Dn \phi = - (\partial_t+U\partial_x)u, ~& \text{on}~\Omega\times (0,T)\\
\Dn \phi= 0, & ~\text{on}~\Theta_1\times(0,T)\\
\psi=\phi_t+U\phi_x=0, &~ \text{on}~\Theta_2\times(0,T)\end{cases}
\end{equation}
where 
\begin{eqnarray}  \label{free}
\mathcal{B}_1 u \equiv \Delta u + (1-\mu) B_1 u, 
\notag \\
\mathcal{B}_2 u \equiv \partial_{\nu} \Delta u + (1-\mu) B_2 u - \mu_1 u -D(u_t)= 0 ~~\mathrm{on}~~ \Gamma_1,  \notag 
\end{eqnarray}
and we have partitioned the boundary $\partial \Omega = \partial\Omega_1 \sqcup \partial\Omega_2$. The boundary operators $B_1$ and $B_2$ are given by \cite{lagnese}:
\begin{equation*}
\begin{array}{c}
B_1u = 2 \nu_1\nu_2 u_{xy} - \nu_1^2 u_{yy} - \nu_2^2 u_{xx}=-\partial_{\tau}^2u-\nabla \cdot \nu(\xb)\Dn u\;, \\
\\
B_2u = \partial_{\tau} \left[ \left( \nu_1^2 -\nu_2^2\right) u_{xy} + \nu_1
\nu_2 \left( u_{yy} - u_{xx}\right)\right]\,=\partial_{\tau}\partial_{\nu}\partial_{\tau}u,%
\end{array}%
\end{equation*}
where $\nu=(\nu_1, \nu_2)$ is the outer normal to $\Gamma$, $\tau= (-\nu_2,
\nu_1)$ is the unit tangent vector along $\partial\Omega$. The parameter $\mu_1$
is nonnegative; the constant $0<\mu<1$ has the meaning of the
Poisson modulus. The abstract boundary damping is encapsulated in the term $D(\cdot)$. 
The regions  are described by the picture below.

\includegraphics[scale=.24]{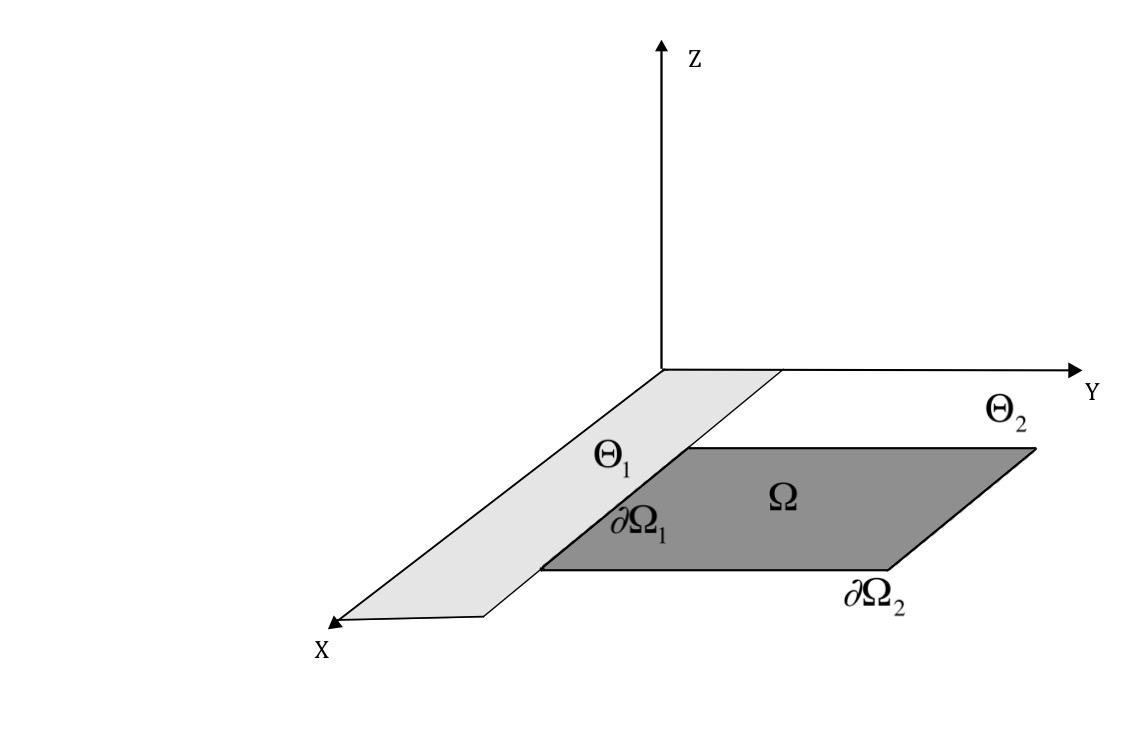}

\begin{remark}
The configuration above arises in the study of airfoils, but another related configuration referred to as {\em axial flow} takes the flow to occur in the $y$ direction in the picture above. In our analysis, the geometry of the plate (and hence the orientation of the flow) do not play a central role. In practice, the orientation can have a dramatic effect on the occurrence and magnitude of the oscillations associated with the flow-structure coupling. In the case of axial flow, the above configuration is often discussed in the context of {\em flag flutter} or flapping. See \cite{K4} for more details. 
\end{remark}

The physical nature of the models given by the boundary conditions in \eqref{physical} makes their analysis desirable; however, such models involve a high degree of mathematical complexity due to the dynamic and mixed nature of the boundary coupling near the flow-plate interface. From the point of view of the existing analysis, much of the well-posedness and long-time behavior analysis is contingent upon taking Clamped boundary conditions assumed for the plate; these allow for smooth extensions to $\R^2$ of the Neumann flow boundary conditions
satisfied by the flow. In the absence of these, one needs to approximate the original dynamics in order to construct
 sufficiently smooth functions amenable to PDE calculations.  This is a technical challenge and was carried out in a similar fashion in \cite{supersonic}, although the need for this analysis was not due to plate boundary conditions.  
 
 \section{Acknowledgements} 
The authors  are  grateful  to  Prof. A.V. Balakrishnan for  a longtime   sharing of  his pioneering results in the field  of continuum aeroleasticity. Additionally the authors are grateful to Prof. E. Dowell for very constructive and informative ongoing discussion regarding flow-structure interaction models and configurations of recent interest.

The research conducted by Irena Lasiecka was supported by the grants NSF-DMS-0606682 and AFOSR-FA99550-9-1-0459.

\section{Appendix: Finite Hilbert Transform}\label{hilbert}
In this appendix we present results on the finite Hilbert transform which are critical to the analysis in \cite{bal0,bal4} and to our analysis when $\Omega$ is an interval. Our references are \cite{FH, FH1,FH2}. The results provide the motivation for assumptions concerning the additional trace regularity found in Condition \ref{le:FTR0}, and hence are apropos to the invertibility of the higher dimensional Hilbert-like singular integral transform. 

We consider the case where $\Omega=(-1,1)$. For $f\in L_p(-1,1)$, define the finite Hilbert Transform to be:
$$g(x)=\cT(f)=\frac{1}{\pi}PV\int_{-1}^1\dfrac{f(y)}{x-y} dy,$$ which is $\mathscr{L}\left(L_p(\Omega)\right)$, $p \in (1,\infty)$.
We are concerned with the inversion formula, as given in Tricomi \cite{tricomi}:
$$f(x)=-\frac{1}{\pi}PV \int_{-1}^1 \sqrt{\dfrac{1-y^2}{1-x^2}}\dfrac{g(y)}{y-x}dy+\dfrac{C}{\sqrt{1-x^2}},$$ noting that the null space of $\cT$ is captured by the last term above. 

Using Tricomi's algebraic identity:
$$\sqrt{\dfrac{1-y^2}{1-x^2}}\dfrac{1}{y-x}=\dfrac{1}{x-y}+\dfrac{x+y}{\sqrt{1-x^2}\big(\sqrt{1-x^2}+\sqrt{1-y^2}},$$
we rewrite the inversion formula as
$$f(x)=-\frac{1}{\pi}\dfrac{1}{\sqrt{1-x^2}}\int_{-1}^1\dfrac{(x+y)}{\sqrt{1-x^2}+\sqrt{1-y^2}}g(y)dy+\cT(g)+\dfrac{C}{\sqrt{1-x^2}}.$$

With regard to \cite{tricomi}, we have a straightforward argument (from the inversion formula) which yields the {\em lowest possible integrability} of $g$ in order for the inversion formula to be valid. This yields a certain type of {\em optimality} for the inversion:
For $g \in L_{(4/3)^+}$, we get $f \in L_{(4/3)^-}.$

In addition, the following more modern theorem is valid \cite{FH2}:

\begin{theorem}\label{invhilb} Consider the finite Hilbert transform on the unit interval $\Omega = (-1,1)$ 

$H: L_p(\Omega) \to L_p(\Omega)$. Then:
\begin{enumerate}
\item For any $ p \in (1,2) $ the map $H$ is Fredholm of index $1$ and the psedoinverse $H^{\# } $ is bounded on $ L_p(\Omega) $. 
\item For any $p \in (2, \infty ) $ the map $H$ is injective and Fredholm of index $-1$. Thus, the inverse is bounded on its range, where $$\mathscr{R}(H) =\big\{ f \in L_p(\Omega): \int_{-1}^1 \dfrac{f}{\sqrt{1-x^2}} dx =0\big\} .$$ 
\item  For $p=2$, $\mathscr{R}(H)$ is dense and proper.  
\end{enumerate}
\end{theorem}



\end{document}